\documentclass[a4paper]{article}
\RequirePackage{fix-cm}

\makeatletter
\DeclareRobustCommand{\em}{\@nomath\em\if b\expandafter\@car\f@series\@nil\bfseries\else\itshape\fi}
\makeatother

\def\TF{black!4}
\def\RF{black!48}
\def\FF{black!16}

\def\vsigma{\vec\sigma}
\def\valpha{\vec\alpha}
\def\vbeta{\vec\beta}
\newcommand{\vgamma}{\vec\gamma}

\newcommand{\vTheta}{\vec\Theta}
\newcommand{\vXi}{\vec\Xi}

\def\vd{\vec d}
\def\vk{\vec k}
\newcommand{\vh}{\vec h}
\newcommand{\vr}{\vec r}

\newcommand{\va}{\vec a}
\newcommand{\vb}{\vec b}
\newcommand{\vc}{\vec c}

\newcommand{\ve}{\vec e}

\newcommand{\vw}{\vec w}
\newcommand{\vz}{\vec z}
\newcommand{\vs}{\vec s}

\newcommand{\vx}{\vec x}
\newcommand{\vy}{\vec y}

\newcommand{\vi}{\vec i}

\usepackage{hyperref}
\def\texorpdfstring#1#2{#1}

\usepackage[utf8]{inputenc}
\usepackage[english]{babel}

\usepackage{amssymb}
\usepackage{amsmath}
\usepackage{amsthm}
\newtheorem{definition}{Definition}[section]
\newtheorem{proposition}[definition]{Proposition}
\newtheorem{lemma}[definition]{Lemma}
\newtheorem{theorem}[definition]{Theorem}
\newtheorem{corollary}[definition]{Corollary}
\theoremstyle{remark}\newtheorem{remark}[definition]{Remark}
\theoremstyle{remark}\newtheorem{example}[definition]{Example}

\usepackage{mathrsfs}
\usepackage{bm}

\usepackage{graphicx}
\usepackage{tikz}
\usetikzlibrary{calc,positioning,matrix,arrows,arrows.meta,external,decorations,decorations.pathreplacing,decorations.shapes,shapes.misc,patterns}
\tikzsetexternalprefix{fig_src/}

\tikzset{shape cross style/.style= {color = red!90!black, draw, line width = .09cm, inner xsep = 0.06cm, inner ysep = 0.06cm} }

\usepackage{enumitem}
\newcounter{lastH}
\newcounter{subH}
\newlist{hlist}{enumerate}{1}
\setlist[hlist]{label={H\arabic*},resume=hlist}

\newcounter{lastE}
\newlist{elist}{enumerate}{1}
\setlist[elist]{label={E\arabic*},resume=elist}

\makeatletter
\newcommand{\hitem}[2]{%
\stepcounter{lastH}%
\newcounter{#1}\setcounter{#1}{0}%
\expandafter\gdef\csname L#1\endcsname{#2}%
\expandafter\xdef\csname N#1\endcsname{\thelastH}%
\item[${\bf\mathrm{H}}^{\bf\thelastH}_{#2}$]%
\def\@currentlabel{{${\bf\mathrm{H}}^{\bf\thelastH}_{#2}$}}\label{H:#1}%
}
\newcommand{\hsubitem}[3]{%
\stepcounter{#3}%
\item[${\bf\mathrm{H}}^{\bf\csname N#3\endcsname\Alph{#3}}_{#2}$]%
\def\@currentlabel{{${\bf\mathrm{H}}^{\bf\csname N#3\endcsname\Alph{#3}}_{#2}$}}\label{H:#1}%
}
\newcommand{\lritem}[2]{\stepcounter{lastH}\item[$\bf\mathrm{LR}^{\bf\thelastH}_{#2}$]\def\@currentlabel{{$\bf\mathrm{LR}^{\bf\thelastH}_{#2}$}}\label{LR:#1}\setcounter{subH}{1}}
\newcommand{\hbitem}[2]{\stepcounter{lastH}\item[$\bf\mathrm{HB}^{\bf\thelastH}_{#2}$]\def\@currentlabel{{$\bf\mathrm{HB}^{\bf\thelastH}_{#2}$}}\label{HB:#1}\setcounter{subH}{1}}

\newcommand\elabel[1]{\def\@currentlabel{E\thelastE}\label{#1}\tag{E\thelastE}\stepcounter{lastE}}
\makeatother

\newlength{\hatchspread}
\newlength{\hatchthickness}
\newlength{\hatchshift}
\newcommand{\hatchcolor}{}
\tikzset{hatchspread/.code={\setlength{\hatchspread}{#1}},
         hatchthickness/.code={\setlength{\hatchthickness}{#1}},
         hatchshift/.code={\setlength{\hatchshift}{#1}},
         hatchcolor/.code={\renewcommand{\hatchcolor}{#1}}}
\tikzset{hatchspread=3pt,
         hatchthickness=0.4pt,
         hatchshift=0pt,
         hatchcolor=black}

\pgfdeclarepatternformonly[\hatchspread,\hatchthickness,\hatchshift,\hatchcolor]
   {custom north east lines}
   {\pgfqpoint{\dimexpr-2\hatchthickness}{\dimexpr-2\hatchthickness}}
   {\pgfqpoint{\dimexpr\hatchspread+2\hatchthickness}{\dimexpr\hatchspread+2\hatchthickness}}
   {\pgfqpoint{\dimexpr\hatchspread}{\dimexpr\hatchspread}}
   {
    \pgfsetlinewidth{\hatchthickness}
    \pgfpathmoveto{\pgfqpoint{\dimexpr\hatchshift-0.15pt}{-0.15pt}}
    \pgfpathlineto{\pgfqpoint{\dimexpr\hatchspread+0.15pt}{\dimexpr\hatchspread-\hatchshift+0.15pt}}
    \ifdim \hatchshift > 0pt
      \pgfpathmoveto{\pgfqpoint{-0.15pt}{\dimexpr\hatchspread-\hatchshift-0.15pt}}
      \pgfpathlineto{\pgfqpoint{\dimexpr\hatchshift+0.15pt}{\dimexpr\hatchspread+0.15pt}}
    \fi
    \pgfsetstrokecolor{\hatchcolor}
    \pgfusepath{stroke}
   }

\pgfdeclarepatternformonly[\hatchspread,\hatchthickness,\hatchshift,\hatchcolor]
   {custom north west lines}
   {\pgfqpoint{\dimexpr-2\hatchthickness}{\dimexpr-2\hatchthickness}}
   {\pgfqpoint{\dimexpr\hatchspread+2\hatchthickness}{\dimexpr\hatchspread+2\hatchthickness}}
   {\pgfqpoint{\dimexpr\hatchspread}{\dimexpr\hatchspread}}
   {
    \pgfsetlinewidth{\hatchthickness}
    \pgfpathmoveto{\pgfqpoint{0pt}{\dimexpr\hatchspread+\hatchshift}}
    \pgfpathlineto{\pgfqpoint{\dimexpr\hatchspread+0.15pt+\hatchshift}{-0.15pt}}
    \ifdim \hatchshift > 0pt
      \pgfpathmoveto{\pgfqpoint{0pt}{\hatchshift}}
      \pgfpathlineto{\pgfqpoint{\dimexpr0.15pt+\hatchshift}{-0.15pt}}
    \fi
    \pgfsetstrokecolor{\hatchcolor}
    \pgfusepath{stroke}
   }

\usepackage{zref-savepos}
\newcounter{NoTableEntry}
\renewcommand*{\theNoTableEntry}{NTE-\the\value{NoTableEntry}}

\newcommand*{\notableentry}{%
  \multicolumn{1}{@{}c@{}|}{%
    \stepcounter{NoTableEntry}%
    \vadjust pre{\zsavepos{\theNoTableEntry t}}
    \vadjust{\zsavepos{\theNoTableEntry b}}
    \zsavepos{\theNoTableEntry l}
    \hspace{0pt plus 1filll}%
    \zsavepos{\theNoTableEntry r}
    \tikz[overlay]{%
      \draw[red]
        let
          \n{llx}={\zposx{\theNoTableEntry l}sp-\zposx{\theNoTableEntry r}sp},
          \n{urx}={0},
          \n{lly}={\zposy{\theNoTableEntry b}sp-\zposy{\theNoTableEntry r}sp},
          \n{ury}={\zposy{\theNoTableEntry t}sp-\zposy{\theNoTableEntry r}sp}
        in
        (\n{llx}, \n{lly}) -- (\n{urx}, \n{ury})
      ;
    }%
  }%
}

\newcommand{\R}{\mathbb{R}}
\newcommand{\C}{\mathbb{C}}
\newcommand{\N}{\mathbb{N}}
\newcommand{\Z}{\mathbb{Z}}

\newcommand{\aop}{\aleph}
\newcommand{\paop}{\Pi}

\newcommand{\abs}[1]{\left\lvert {#1} \right\rvert}
\newcommand{\cabs}[1]{\lvert {#1} \rvert}
\newcommand{\babs}[1]{\big\lvert {#1} \big\rvert}

\newcommand{\measure}[1]{\mu({#1})}
\newcommand{\cardinality}[1]{\# {#1} }

\newcommand{\norm}[1]{\left\lVert {#1} \right\rVert}
\newcommand{\cnorm}[1]{\lVert {#1} \rVert}
\newcommand{\bnorm}[1]{\big\lVert {#1} \big\rVert}
\newcommand{\Bnorm}[1]{\Big\lVert {#1} \Big\rVert}
\newcommand{\onorm}[1]{|\kern-0.25ex|\kern-0.25ex|#1|\kern-0.25ex|\kern-0.25ex|}

\newcommand{\ceil}[1]{\left \lceil{#1}\right \rceil}
\newcommand{\floor}[1]{\left \lfloor{#1}\right \rfloor}

\newcommand{\thus}{\ensuremath\Rightarrow}
\def\iff{\ensuremath\Leftrightarrow}

\DeclareMathOperator*{\esssup}{ess\,sup}

\DeclareMathOperator{\support}{supp}

\DeclareMathOperator*{\maxv}{\mathbf{max}}
\DeclareMathOperator*{\minv}{\mathbf{min}}

\DeclareRobustCommand{\icircle}{\tikz [anchor=base, baseline]{\draw (0,0) circle (.25ex);}}

\newcommand{\interior}[1]{\overset{\icircle}{#1}}
\newcommand{\closure}[1]{\overline{#1}}

\DeclareMathOperator{\diam}{diam}
\DeclareMathOperator{\esupp}{{\mathtt s}}
\DeclareMathOperator{\SPAN}{span}
\DeclareMathOperator{\bbox}{\mathtt b}

\usepackage{bbm}
\newcommand{\ind}{{\mathbbm{1}}} 

\def\boundary{\partial}

\newcommand{\mesh}{\mathcal{M}}
\DeclareMathOperator{\extension}{ext}

\renewcommand{\vec}[1]{\bm{#1}}

\DeclareFontFamily{U}{mathx}{\hyphenchar\font45}
\DeclareFontShape{U}{mathx}{m}{n}{
  <5> <6> <7> <8> <9> <10>
  <10.95> <12> <14.4> <17.28> <20.74> <24.88>
  mathx10
}{}
\DeclareSymbolFont{mathx}{U}{mathx}{m}{n}
\DeclareMathSymbol{\bigtimes}{1}{mathx}{"91}

\newcommand{\blossom}{\mathfrak c}

\newcommand{\poly}{\mathbb{P}}    
\newcommand{\spline}{\mathbb{S}}
\newcommand{\ppoly}{\mathbb{K}}

\newcommand{\contained}{\mathtt M}
\newcommand{\actives}{\mathtt A}
\newcommand{\actived}{\mathtt E}
\newcommand{\smooth}{\mathcal{C}}

\newcommand{\base}[1]{#1_B}
\newcommand{\vmax}[1]{\max #1}
\newcommand{\vmin}[1]{\min #1}

\DeclareRobustCommand{\pmesh}{\tikz \draw[very thin,step=.075] (0,0) grid (.15,.15);}

\DeclareRobustCommand{\pshapeO}{\tikz{ \node[fill=gray!05,draw=black,very thin,rectangle, inner sep=.9]{\small$\omega$\!};}}
\DeclareRobustCommand{\pshapeA}{\tikz{ \node[fill=gray!05,draw=black,very thin,rectangle, inner sep=.9]{\small$\esupp$};}}
\DeclareRobustCommand{\Gconcave}{\tikz{\fill (0,0) -- (1ex,0)--(1ex,.5ex)--(.5ex,.5ex)--(.5ex,1ex)--(0ex,1ex);}}

\def\Ctot{C}

\def\Capprox{C_{\Pi}}
\def\Clambda{C_{\lambda}}
\def\Cphi{C_{\Phi}}
\def\Cesupp{C_{\esupp}}
\def\Cmesh{C_{\pmesh}}
\def\Cactives{C_{\actived}}
\def\Cduals{C_{\actived}}
\def\Cnum{C_{\#}}

\def\Clength{C_{\mathrm m}}
\def\CisotropicA{C_{\pshapeA}}
\def\CisotropicO{C_{\pshapeO}}
\def\Cisotropic{\CisotropicA}
\def\GrowthTAG{e}
\def\Cgrowth{C_\GrowthTAG}
\def\Bgrowth{B_\GrowthTAG}
\def\Egrowth{{S_\GrowthTAG}}

\def\Cknots{C_\Xi}
\def\Cgraded{C_g}

\def\he{\vec h_\eta}
\def\hei#1{h_{\eta,#1}}

\def\hz{\vec h_{\zeta_\phi}}

\def\ho{\vec h_\omega}
\def\hoi#1{h_{\omega,#1}}

\def\hp{\vec h_\phi}
\def\hpi#1{h_{\phi,#1}}

\def\ha{\vec h_{\esupp_\phi}}
\def\hai#1{h_{\esupp_\phi,#1}}

\def\hmd{\mathtt{h}_{\Phi_d}}
\def\hm{\mathtt{h}_{\Phi}}
\def\hmi#1{\mathtt{h}_{\Phi,#1}}

\def\hof{\mathtt{h}_\mesh}
\def\hofi#1{\mathtt{h}_{\mesh,#1}}

\begin{document}
\title{Local approximation operators on box meshes}
\author{Andrea Bressan \and Tom Lyche}


\maketitle

\begin{abstract}
This paper analyzes the approximation properties of spaces of piecewise tensor product polynomials over box meshes with a focus on application to IsoGeometric Analysis (IGA).
Local and global error bounds with respect to Sobolev or reduced seminorms are provided.
Attention is also paid to the dependence on the degree and exponential convergence is proved for the approximation of analytic functions in the absence of non-convex extended supports.
\end{abstract}

\section{Introduction}

Let $\Omega$ be a union of axis aligned boxes in $\R^n$ and $\spline$ a space of piecewise polynomials of degree $\vd:=(d_1,\dots, d_n)$ on $\Omega$. We use a local approximation operator $\aop:L^p(\Omega)\to \spline$, which reproduces polynomials of degree $\vd$, to derive a priori approximation estimates for how well a function $f$ can be approximated by a function in $\spline$. Bounds for the approximation error on both the mesh elements and on $\Omega$ are provided.
The $L^p$, $L^q$ estimates on $\Omega$ have the following form 
\begin{equation}
\label{eq:abstract-form}
\norm{\partial^{\vsigma}(f-\aop f)}_{p}\le\Ctot\sum_{\vk\in K}\cnorm{\rho_{\vk}\partial^{\vk} f}_q,
\end{equation}
where $K\subset\N^n $ is an index set of integers and the weights $\rho_{\vk}$ are powers of the local resolution of $\spline$. The precise form of the $\rho_{\vk}$ depends on $K$, $\vk$, $p$, $q$, $\vd$ and ${\vsigma}$.
The constant $\Ctot$ depends on a subset of some constant $\Clambda, \dots, \CisotropicA$, which come from the abstract assumptions \ref{H:polynomial},\dots,\ref{H:isotropicA}; see Table \ref{table:hypothesis} at the end of the paper.

Different index sets $K$ are possible, we focus on $K=\{\vk\in\N^n:\abs{\vk}=d+1\}$ that corresponds to Sobolev seminorms and on index sets $K$ that corresponds to \emph{reduced} seminorms that involve a smaller set of partial derivatives.
For $\vsigma=\vec0$ the reduced seminorm leads to bounds in term of the partial derivatives in the coordinate directions
\begin{equation}
\label{eq:abstract-form-reduced}
\bnorm{f-\aop f}_{p}\le\Ctot\sum_{i=1}^n\bnorm{\rho_{i}\partial^{k_i}_i f}_q.
\end{equation}

For Sobolev seminorms and $p\le q$ we are able to weaken the usual \emph{mesh quasi uniformity} and \emph{mesh shape regularity} assumptions matching those for tensor product splines. This can be seen by comparing our assumptions \ref{H:mesh-num} and \ref{H:isotropicA} with the corresponding assumptions in \cite{MR3585787}\cite{MR3627466}\cite{MR3439218}.

For $1\le p\le q$ and $\vsigma=\vec0$ we obtain \emph{anisotropic} estimates that take into consideration the local resolution of $\spline$ in the different coordinate directions. In the other cases the estimates are \emph{isotropic}, i.e., the directional information is discarded.

We also consider a sequence of spline spaces $\{\spline_d\}$ of increasing degree $\vd=(d, \dots, d)$ and corresponding approximation operators $\{\aop_d\}$.
Under suitable assumptions, if $f:\Omega\to\R$ admits an analytic extension on $\Delta\subset\C^n$, the approximation error decreases exponentially as $d$ increases. There are $\tau, C_{\tau}, \tau_{\#},C_{\tau_{\#}}\in\R$ that depend on $\Delta$, ${\vsigma}$ and $n$, but independent of the degree $d$, such that
\begin{equation}\label{eq:INTRO-exponential}
 \begin{aligned}
\norm{\partial^{\vsigma} (f-\aop_df)}_\infty \le& C_{\tau} \tau^{d+1-\abs{\vsigma}}\norm{f}_{\infty, \Delta}
\\\le& C_{\tau_{\#}} \tau_{\#}^{(\dim\spline_d)^{\frac1n}}\norm{f}_{\infty, \Delta}.
\end{aligned}
\end{equation}
Exponential convergence is of interest for the numerical solutions of PDEs and it is known to hold for Finite Element discretizations and univariate splines even in the presence of singularities \cite{MR3379755}\cite{MR3351174}. The novelty here is that exponential convergence is proved using local approximation operators and without assumptions on the smoothness of the functions in $\spline_d$.

Following the application of splines to numerical methods for PDEs, mostly in the framework of IsoGeometric Analysis (IGA) \cite{MR2152382}, there has been a renewed interest in extension of tensor product splines (TPS) that are suitable to adaptive methods.
Many of the available constructions fit our abstract framework. We provide an approximation operator $\aop$ satisfying \ref{H:polynomial},\dots,\ref{H:isotropicA} for TPS, Analysis Suitable T splines (AST) \cite{MR3084741}\cite{MR3003061}, truncated hierarchical splines (THB) \cite{MR2925951} and Locally Refined splines (LR) \cite{MR3019748}.

TPS are included for the sake of comparison and similar results are available in the literature \cite{local-spline-approx}\cite{ANU:9260759}\cite{MR2877950}.
The results for AST splines and for THB splines require weaker assumptions on the mesh compared to \cite{MR3627466}\cite{MR3439218}\cite{MR3003061}; moreover, both global estimates and estimates without mixed derivatives are included.
To the best of our knowledge the operator in Section~\ref{sec:lr-splines} is the first quasi-interpolant that is proposed for LR splines.

\subsection{Summary}

The notation and the setting of this paper are described in Section~\ref{sec:notation}.
Section~\ref{sec:framework} contains the main approximation results under the assumptions \ref{H:polynomial},\dots,\ref{H:isotropicA}.
The index set $K$ of partial derivatives is determined by \ref{H:approx} that describes polynomial approximation.
In Section~\ref{sec:TA} we prove that different variants of averaged Taylor expansion operators satisfy \ref{H:approx}.
Section~\ref{sec:estimates} specialises the abstract theory to bounds in terms of Sobolev and \emph{reduced} seminorms.
Section~\ref{sec:EXP} studies the approximation of analytic functions and contains the proof of exponential convergence.
Section~\ref{sec:b-splines} recalls the B-spline theory necessary for this paper and provides the building blocks for the approximation operators $\aop$ described in Section~\ref{sec:SP}.
Sections~\ref{sec:SP} contains an approximation operator for TPS, AST, THB and LR. For each case we provide a definition of $\aop$ that satisfies the required abstract assumptions.
Section~\ref{sec:remarks} contains a few closing remarks.

\section{Preliminaries}
\label{sec:notation}

\subsection{Notation}
Sets are usually denoted by capital letters, except when they are subsets of a topological space.
The cardinality of a set $A$ is denoted by $\cardinality{A}$.
The interior of a subsets $\omega$ of a topological space, i.e.\ the biggest open set contained in $\omega$, is denoted by $\interior{\omega}$ and the closure, i.e.\ the smallest closed set containing $\omega$, is denoted by $\closure\omega$.

$\Z$ is the set of integers and $\N:=\{0, 1, \dots\}$. The floor, or integer part, of a real number $a\in\R$ is denoted by $\floor{a}:=\max\{z\in\Z: z\le a\}$
. We also use the positive and negative parts: $a_+:=\max\{0, a\}$ and $a_-:=\min\{0,a\}$ so that $a= a_++a_-$.

Multi-indices in $\N^n$ and vectors in $\R^n$ are highlighted by boldface. For convenience we use $\vec0:=(0, \dots, 0)$ and $\vec1:=(1, \dots, 1)$ in $\N^n$ and $\R^n$.
Given a multi-index ${\valpha}\in\N^n$ and a sufficiently smooth function $f:\Omega\subseteq\R^n\to\R$
$$\partial^{\valpha} f:=\frac{\partial^{\alpha_1+\dots+\alpha_n}}{\partial x_1^{\alpha_1}\cdots\partial x_n^{\alpha_n}}f. $$
The factorial of a multi-index $\valpha \in\N^n$ is $\valpha!:=\prod_{i=1}^n (\alpha_i!)$. Consequently for ${\valpha}, \, {\vbeta}\in\N^n$ we have
$$
\binom {\valpha}{\vbeta}:=\frac{\valpha!}{(\valpha-\vbeta)!\vbeta!} =\prod_{i=1}^n\binom {\alpha_i}{\beta_i}.
$$
Many scalar operations are extended to vectors in $\R^n$.
The relations $<, >, \ge, \le$ hold on a pair of vectors $\va=(a_1, \dots, a_n), \vb=(b_1, \dots, b_n)\in\R^n$ if and only if they hold for each pair of components, e.g.\ $$
\va\le\vb\iff a_i\le b_i, \, \forall i=1, \dots, n.
$$
Similarly $\maxv\{\va, \vb\}$, $\minv\{\va, \vb\}$, and $\va_\pm$ act component wise: e.g.,
\begin{align*}
\va_+&:=((a_1)_+, \dots, (a_n)_+), &&&\maxv\{\va, \vb\}&:=(\max\{a_1, b_1\}, \dots, \max\{a_n, b_n\}).
\end{align*}
For a vector $\va\in\R^n$ we define
\begin{align*}
&\abs{\va}:=\sum_{i=1}^n\abs{a_i}, &&
\norm{\va} :=\big(\sum_{i=1}^n a_i^2\big)^{\frac12}, \\
&\vmax {\va}:=\max\{a_1, \dots, a_n\}, &&\vmin\va :=\min\{a_1, \dots, a_n\},
\end{align*}
and for $\va, \vb\in\R^n$ we define the power $\va^{\vb}$ as
\begin{align*}
\va^{\vb} &:=\prod_{i=1}^n a_i^{b_i}.
\end{align*}
Note that we have the following relations for $\va,\vb,\vc\in\R^n$
\begin{equation}\label{eq:power-identity}
	\frac{\va^{\vc}}{\vb^{\vc}}=\Big(\frac{a_1}{b_1},\dots,\frac{a_n}{b_n}\Big)^{\vc}, \qquad \va\ge \vec0\thus \va^{\vb}\le (\vmax \va)^{\abs{\vb_+}}(\vmin\va)^{-\abs{\vb_-}}.
\end{equation}

As usual $\smooth(U)$ denotes the space of continuous functions $U\to\R$ and for ${\alpha}\in\N$, $\smooth^{\alpha}(U)$ is the space of continuous functions whose derivatives up to order $\alpha$ are in $\smooth(U)$.
The definition is extended to ${\valpha}\in\N^n$ with the following meaning
$$
\smooth^{{\valpha}}(U):=
\{f:U\to\R\, :\forall {\vbeta}\in\N^n\text{ with } {\vbeta}\le {\valpha}, \partial^{\vbeta} f\in\smooth(U)\}.
$$

Polynomials of maximal degree $d$ and polynomials of multidegree $\vd=(d_1, \dots, d_n)$ are respectively denoted by
\begin{align*}
\poly_d&:=\SPAN\{\vx^{\valpha}:\, {\valpha}\in\N^n, \, \abs{\valpha}\le d\}, \qquad
\poly_{\vd}&:=\SPAN\{\vx^{\valpha}:\, {\valpha}\in\N^n, \, {\valpha}\le\vd\}.
\end{align*}
In general the polynomial space $\poly_A$ associated with $A\subseteq\N^n$ is
$$\poly_A:=\SPAN\{\vx^{\valpha}:\, {\valpha}\in A\}.$$

The standard Lebesgue measure on $\R^n$ is denoted by $\mu$. As customary $L^p(U)$, $1\le p\le\infty$, is the Banach space of the equivalence classes of measurable functions $U\to\R$ that agree almost everywhere (a.e.) and $L^p(U;\R^n)$ is the corresponding space of vector valued functions. The norm on $L^p(U)$ is the usual
$$\norm{f}_{p, U}:=\begin{cases}(\int_U\abs{f(\vx)}^p\, d\vx )^{\frac1p}& p\neq\infty\\\esssup\{\abs{f(\vx)}, \, \vx\in U\}&p=\infty.\end{cases}$$
The dual of $L^p(U)$ is the space of linear continuous functionals $\lambda:L^p(U)\to\R$ and it is denoted by $L^p(U)^*$. It is a Banach space with norm
$$
\norm{\lambda}_{*p}:=\sup\{\lambda(f):\, f\in L^p(U):\, \norm{f}_p=1\}.
$$
If $1<p<\infty$ then $L^p(U)^*$ is isometrically isomorphic to $L^{p'}(U)$ where $p'=(1-1/p)^{-1}$ and the isomorphism maps $\lambda$ to $w$ if for all $f\in L^p(U)$
$$
\lambda (f)=\int_U f(\vx) w (\vx)\, d\vx.
$$
The support of a function $f\in L^p(U)$ is the closed set
$$\support f := U\setminus\bigcup_{W\in Z} W, \qquad Z=\{\text{open } W\subseteq U: f|_W=0\ \text{a.e.}\}.$$
This definition agrees with the standard definition $\support f=\closure{\{\vx\in U:\, \tilde f(\vx)\ne 0\}}$ if $\tilde f$ is a continuous representative of $f$.
For $\lambda\in L^p(U)^*$ we use $\support\lambda$ to denote the distributional support, i.e.\
$$\support\lambda:=U\setminus\bigcup_{W\in Z} W, \ Z=\{\text{open }W\subseteq U:\forall f\text{ with }\support f\subset W, \, \lambda(f)=0\}.$$
For $1<p<\infty$ the supports of $\lambda$ as an element of $L^p(U)^*$ and as an element of $L^{p'}(U)$ coincide.


\subsection{Setting}

Spaces of piecewise polynomials over box meshes are the main subject of this paper.
Here a box in $\R^n$ is axis aligned, i.e.\ it is either empty or a Cartesian product of non-empty closed intervals $\eta=[\va,\vb]:=[a_1, b_1]\times\dots\times[a_n, b_n]$.
A box mesh $\mesh$ is a collection of boxes having non-empty pairwise-disjoint interiors
\begin{align*}
\omega\in\mesh &\thus\interior\omega\ne\emptyset, &&&
\omega, \eta\in\mesh,\,\omega\ne\eta &\ \thus\ \interior{\omega}\cap\interior{\eta}=\emptyset.
\end{align*}
The domain $\Omega$ of the box mesh $\mesh$ is the union of the boxes it contains
$$
\Omega:=\bigcup_{\omega\in\mesh}\omega.
$$
Note that the domain of a box mesh is not necessarily connected.

A spline space $\spline$ of degree $\vd$ on a box mesh $\mesh$ is a subspace of
$$\{f:\Omega\to\R:\, \forall\omega\in\mesh, \, \exists g\in\poly_{\vd}:\, f=g\text{ a.e. in }\omega\}.$$
We always assume $\poly_{\vd}\subseteq\spline$.
For continuous functions the above could be simplified to $\{f:\Omega\to\R:\, \forall\omega\in\mesh, \, f|_\omega\in\poly_{\vd}\}$.
Note that the support of any function $f\in\spline$ is a union of boxes in $\mesh$.

Operators $\aop:L^p(\Omega)\to\spline$ can be constructed using a generating system $\Phi$ and a collection of linear functionals $\Lambda=\{\lambda_\phi, \, \phi\in\Phi\}\subset L^p(\Omega)^*$ by setting
\begin{equation}
\label{eq:operator-in-gen-form}
\aop f:=\sum_{\phi\in\Phi}\lambda_\phi(f)\phi.
\end{equation}
Note that $\aop$ is linear 
, i.e., $\aop( a f +b g)=a \aop f + b \aop f$.
For the purpose of this paper $\mesh$, $\Omega$, $\spline$, $\Phi$, $\aop$, $\Lambda$, $p$, $q$ and ${\vsigma}\le\vd$ appearing in \eqref{eq:abstract-form} and \eqref{eq:operator-in-gen-form} are considered given.
Because of this we will shorten $\norm{\cdot}_{p, \Omega}$ to $\norm{\cdot}_p$, i.e. the domain $\Omega$ will be subsumed in norms.

The support of each generator $\phi$ contains a subset $\contained_\phi$ of the elements in $\mesh$.
Vice versa each $\omega$ is contained in the supports of a subset $\actives_\omega$ of $\Phi$
$$
\contained_\phi:=\{\omega\in\mesh:\omega\subseteq\support\phi\},\qquad\qquad
\actives_\omega :=\{\phi\in\Phi:\omega\subseteq\support\phi\}.
$$
The definition of $\actives_\omega$ extends to $U\subseteq \Omega$ as $\actives_U:=\{\phi\in\Phi:\support\phi\cap\interior U\neq\emptyset\}$.
Usually $\support\lambda_\phi\subseteq\support\phi$, for the general case we introduce the \emph{extended support} $\esupp_{\phi}$ and the \emph{extended active set} $\actived_\omega$
\begin{align*}
\esupp_\phi &:=\bbox(\support\phi\cup\support\lambda_\phi)\cap\Omega,&&&
\actived_\omega &:=\{\phi\in\Phi:\omega\subseteq\esupp_\phi\},
\end{align*}
where $\bbox (U)$ is the  \emph{bounding box} of $U\subset\R^n$ , i.e., the smallest box containing $U$:
$$
\bbox (U):=\min\{\text{box }\eta: U\subseteq\eta\}.
$$
From the definition it follows that $\actived_\omega\supseteq\actives_\omega$ with equality if $\support\lambda_\phi\subseteq\support\phi=\esupp_\phi$.
An example of $\esupp_\phi$ is shown in Fig.~\ref{fig:boxes}.
\begin{figure}\begin{center}
\begin{tikzpicture}[xscale=.7, yscale=.6]
\draw (-3.5,-1.7) rectangle (5.5,4.5) (5.5,1.4) node[right] {$\Omega$};
\draw[fill=\FF]	 (-1.5, -1) rectangle (3, 3.3) node[right] {$\esupp_\phi$};
\draw[pattern=custom north west lines, hatchcolor=black, hatchspread=3\hatchspread] 	 (3, 2.3) --(-.5, 2.3)--(-.5, 1)--(-1.5, 1)--(-1.5, -1)--(3, -1)--(3,2.3) node[midway,right] {$\support\phi$};
\draw[fill=\RF]	 (0, 0) rectangle (1.5, 1) node[right] {$\omega$};
\draw[pattern=custom north east lines, hatchcolor=black, hatchspread=3\hatchspread]	 (-1.5, 1.7) rectangle (.3, 3.3) node[above] {$\support\lambda_\phi$};
\end{tikzpicture}\end{center}
\caption[An example of $\omega$, $\support\phi$, $\support\lambda_\phi$ and $\esupp_\phi$ with $\phi\in\actives_\omega$.]{An example of $\omega$ (\tikz\fill[\RF] (0, 0) rectangle (.3, .3);), $\support\phi$ (\tikz\fill[pattern=custom north west lines, hatchcolor=black, hatchspread=2\hatchspread] (0, 0) rectangle (.3, .3);), $\support\lambda_\phi$ (\tikz\fill[pattern=custom north east lines, hatchcolor=black, hatchspread=2\hatchspread] (0, 0) rectangle (.3, .3);) and $\esupp_\phi$ (\tikz\fill[\TF] (0, 0) rectangle (.3, .3);).
}\label{fig:boxes}
\end{figure}

Note that $\esupp_\phi$ is convex and a box if and only if $\bbox(\support\phi\cup\support\lambda_\phi)\subseteq\Omega$.
In this paper we assume that the $\esupp_\phi$ is a \emph{truncated} box, i.e.\ it has the form
\begin{equation}\label{eq:def-truncated-box}
\esupp_\phi=\closure{\eta\setminus \beta},
\end{equation}
where $\eta, \beta$ are two boxes in $\R^n$ sharing a common vertex and $\beta\subset\eta$.
The truncated boxes in $\R^n$ are classified by the face of $\eta$ that does not intersect $\esupp_\phi$ and has maximum dimension $m$, $m=-1$ when $\esupp_\phi=\eta$.
It can be shown that there are  $3^n-2n$ types of truncated boxes in $\R^n$.
In $\R^2$ a truncated box is either a rectangle or an L-shaped domain, see Fig.~\ref{fig:truncated-box-2D}.
In $\R^3$ it is either a box, a Fichera corner or a product of an L-shape and an interval, see Fig.~\ref{fig:truncated-box-3D}.

\begin{figure}
\centering
\begin{tikzpicture}[scale=1.33]
\draw[fill=\FF] (-1, -1) rectangle (1, 1);
\node at (0, -1.35) {$1$ type, ($m=-1$)};
\def\L{\draw[fill=\FF] (-1, -1)--(0, -1)--(0, 0)--(1, 0)--(1, 1)--(-1, 1)--cycle;}
\foreach\i[evaluate=\i as \r using \i*90, evaluate=\i as \s using \i*3+2] in {1}
	{\begin{scope}[xshift=\s cm, rotate=\r]\L\end{scope}}
\node at (5, -1.35) {$4$ types, ($m=0$)};
\fill[red] (6,1) circle (1.5pt) ;
\end{tikzpicture}
\caption{Truncated boxes in $\R^2$: the box and the L-shaped domains. The face of $\eta$ of dimension $m$ not intersecting $\esupp_\phi$ is highlighted in red. }\label{fig:truncated-box-2D}
\end{figure}
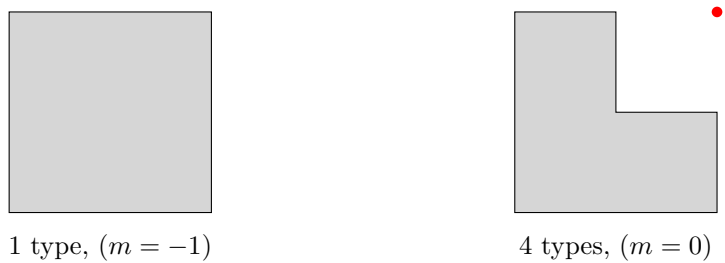
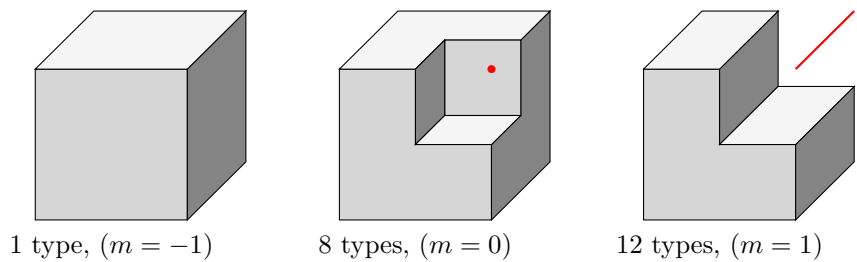
\begin{figure}
\centering
\begin{tikzpicture}[scale=1]
\draw[fill=\FF] (-1, -1, 1)--(1, -1, 1)--(1, 1, 1)--(-1, 1, 1)--cycle;
\draw[fill=\RF] (1, -1, 1)--(1, -1, -1)--(1, 1, -1)--(1, 1, 1)--cycle;
\draw[fill=\TF] (-1, 1, -1)--(1, 1, -1)--(1, 1, 1)--(-1, 1, 1)--cycle;
\node at (0, -1.35, 1) {$1$ type, ($m=-1$)};

\begin{scope}[xshift=4cm]
\draw[fill=\FF] (-1, -1, 1)--(1, -1, 1)--(1, 0, 1)--(0, 0, 1)--(0, 1, 1)--(-1, 1, 1)--cycle;
\draw[fill=\RF] (1, -1, -1)--(1, 1, -1)--(1, 1, 0)--(1, 0, 0)--(1, 0, 1)--(1, -1, 1)--cycle;
\draw[fill=\TF] (-1, 1, -1)--(1, 1, -1)--(1, 1, 0)--(0, 1, 0)--(0, 1, 1)--(-1, 1, 1)--cycle;

\draw[fill=\FF] (0, 0, 0)--(1, 0, 0)--(1, 1, 0)--(0, 1, 0)--cycle;
\draw[fill=\RF] (0, 0, 0)--(0, 1, 0)--(0, 1, 1)--(0, 0, 1)--cycle;
\draw[fill=\TF] (0, 0, 0)--(1, 0, 0)--(1, 0, 1)--(0, 0, 1)--cycle;
\node at (0, -1.35, 1) {$8$ types, ($m=0$)};
\fill[red] (1,1,1) circle (1.5pt) ;
\end{scope}
\begin{scope}[xshift=8cm]
\draw[fill=\FF] (-1, -1, 1)--(1, -1, 1)--(1, 0, 1)--(0, 0, 1)--(0, 1, 1)--(-1, 1, 1)--cycle;
\draw[fill=\RF] (1, -1, 1)--(1, -1, -1)--(1, 0, -1)--(1, 0, 1)--cycle;
\draw[fill=\RF] (0, 0, 1)--(0, 0, -1)--(0, 1, -1)--(0, 1, 1)--cycle;
\draw[fill=\TF] (-1, 1, -1)--(0, 1, -1)--(0, 1, 1)--(-1, 1, 1)--cycle;
\draw[fill=\TF] (0, 0, -1)--(1, 0, -1)--(1, 0, 1)--(0, 0, 1)--cycle;
\node at (0, -1.35, 1) {$12$ types, ($m=1$)};
\draw[thick,red] (1,1,1)  -- (1,1,-1) ;
\end{scope}
\end{tikzpicture}
\caption{Truncated boxes in $\R^3$: the box, the Fichera corner and the product of an L-shape with an interval. The face of $\eta$ of dimension $m$ not intersecting $\esupp_\phi$ is highlighted in red.}\label{fig:truncated-box-3D}
\end{figure}

 \emph{Diversification}, see \cite{MR3268225}, can sometimes provide a $\Phi$ such that all the $\esupp_\phi$'s are truncated boxes, starting from a $\Psi$ not having this property.
It consists in replacing each generator $\phi\in\Psi$ with the set of its restrictions to the connected components of $\support\phi$.
In Fig.~\ref{fig:esupp-phi-outside} there are three examples of a non-convex $\esupp_\phi$ in $\R^2$.
The left is a truncated box. The middle one is not a truncated box, but diversification of $\phi$ can provide two truncated boxes. The right one is not a truncated box and diversification cannot be applied.

\begin{figure}\begin{center}
\def\OC{\FF}
\def\OM{\draw[fill=\OC] (0, 0)--(5, 0)--(5, 4.5)--(3.5, 4.5)--(3.5, 1)--(1.5, 1)--(1.5, 5)--(0, 5)--cycle;}

\begin{tikzpicture}[scale=.6,yscale=.8]
\tikzset{SC/.style={fill=\RF}}
\tikzset{BC/.style={pattern=custom north east lines, hatchcolor=black, hatchspread=3\hatchspread}}

\begin{scope}
\OM
\draw[SC] (2, .5)--(4, .5)--(4, 2)--(3.5, 2)--(3.5, 1)--(2, 1)--cycle node[right] at (4, 1.25) {$\esupp_\phi$};
\draw[BC] (2, .5) rectangle (4, 2);
\end{scope}

\begin{scope}[xshift=7cm]
\OM
\draw[SC] (1, 2) rectangle (1.5, 4)
	 (3.5, 2) rectangle (4, 4) node[right] at (4, 3) {$\esupp_\phi$};
\draw[BC] (1, 2) rectangle (4, 4) ;
\end{scope}

\begin{scope}[xshift=14cm]
\OM
\draw[SC] (1, .5)--(4, .5)--(4, 2)--(3.5, 2)--(3.5, 1)--(1.5, 1)--(1.5, 2)--(1, 2)--cycle node[right] at (4, 1.25) {$\esupp_\phi$};
\draw[BC] (1, .5) rectangle (4, 2);
\end{scope}
\end{tikzpicture}\end{center}
\caption[]{Three examples of a non-convex $\esupp_\phi$(\tikz\fill[\RF] (0, 0) rectangle (.3, .3);), where $\bbox(\esupp_\phi)$(\tikz\fill[pattern=custom north east lines, hatchcolor=black, hatchspread=2\hatchspread] (0, 0)rectangle (.3, .3);) is not contained in the U shaped region $\Omega$(\tikz\fill[\FF] (0, 0) rectangle (.3, .3);).}\label{fig:esupp-phi-outside}
\end{figure}

Given a box $\eta=[\va,\vb]=[a_1, b_1]\times\dots\times[a_n, b_n]$ we define the \emph{size} of $\eta$ as the vector
$$\he:=(\hei 1, \dots, \hei n):=(b_1-a_1, \dots, b_n-a_n)\in\R^n.$$
The definition is extended to $U\subset\R^n$ defining $\vh_U=\vh_{\bbox (U)}$.
Recalling the power notation for vectors it follows that for all $U\subset \R^n$ we have $$\mu(U)\le\vh_U^{\vec1}=\prod_{i=1}^n  h_{U,i}=\measure{\bbox(U)}$$ and equality holds if $\closure{U}$ is a box.

The size of many sets plays a fundamental role: the most recurring are
\begin{align*}
&\ho\text{ for the size of }\omega\in\mesh, \\
&\ha\text{ for the size of }\esupp_\phi, \, \phi\in\Phi, \\
&\hp\text{ as a shortening for }\vh_{\support\phi}\text{: the size of }\bbox(\support\phi).
\end{align*}
We define $\hm$ and $\hof$ as piecewise constant functions in $L^\infty(\Omega, \R^n)$ by
\begin{align}
\label{eq:def-hl}
&\hm|_\omega:=(\hmi 1, \dots, \hmi n) |_\omega:=\maxv_{\phi\in\actived_\omega}\{\ha\}, \\
\label{eq:def-ho}
&\hof|_\omega:=(\hofi 1, \dots, \hofi n)|_\omega:=\ho.
\end{align}
The function $\hm$ gives a local measure of the resolution of $\spline$, while $\hof$ represents the \emph{local mesh size}.
The form of $\rho_{\vk}$ in \eqref{eq:abstract-form} classifies the estimates as
\begin{itemize}
\item \emph{anisotropic}, if $\rho_{\vk}=\hm^{\ve_{\vk}}$ with $\ve_{\vk}\in\R^n$;
\item \emph{isotropic}, if $\rho_{\vk}=(\vmax\hm)^{e_{\vk}}$ with $e_{\vk}\in\R$.
\end{itemize}
Anisotropic estimates differentiate between the size along the coordinate directions, while for isotropic estimates only the maximum, or equivalently the diameter, is considered. For isotropic estimates the equivalence between $\vmax\hm$ and the diameter of $\esupp_\phi$ is
$$
\vmax\hm |_\omega\le\max_{\phi\in\actived_\omega}\{\diam\esupp_\phi\}\le \sqrt n \vmax\hm |_\omega.
$$
When dealing with mixed $p, q$ norms we often need
\def\pq{\nu}
\def\vpq{\vec\nu}
\begin{equation}\label{eq:pq-diff}
\pq:=\frac1p-\frac1q, \qquad\vpq:=(\pq, \dots, \pq),
\end{equation}
and the following inequalities
\begin{equation}\label{eq:sums}\begin{aligned}
\big(\sum_{i=1}^m a_i\big)^e\le\sum_{i=1}^m a_i^e\le m^{1-e}\big(\sum_{i=1}^m a_i\big)^e, && 0\le e\le 1,\\
\sum_{i=1}^m a_i^e\le\big(\sum_{i=1}^m a_i\big)^e\le m^{e-1}\sum_{i=1}^m a_i^e, && 1\le e<\infty,
\end{aligned}
\end{equation}
where $a_1, \dots, a_m\ge 0$.
The left inequalities can be proved by rescaling the $a_i$ by $(\sum_{i=1}^m a_i)^{-1}$ and using that for $a\in[0, 1]$, $a^e\le a\iff e\ge 1$.
The right inequalities are corollaries of H\"older's inequallity: for $1< e <\infty$ and $a_1, \dots, a_m, b_1, \dots, b_m\ge 0$
\begin{equation}\label{eq:holder}
\sum_{i=1}^m a_i b_i\le\Big (\sum_{i=1}^m a_i^e\Big )^{\frac1e}\Big (\sum_{i=1}^m b_i^{\frac e{e-1}}\Big )^{\frac {e-1}e}.
\end{equation}

\section{Estimates on box meshes}
\label{sec:framework}

This section describes how error estimates of the form \eqref{eq:abstract-form} follow from the assumptions \ref{H:polynomial},\dots,\ref{H:isotropicA}.
In the first subsection local error estimates are derived from \ref{H:polynomial},\dots,\ref{H:approx}.
Global error bounds require additionally \ref{H:duals},\dots,\ref{H:isotropicA} and are stated in the second and third subsections.
The fourth subsection describes the dependence of the mesh assumption on the parameters of \ref{H:approx}.

It is important to highlight that error estimates can only be obtained for suitable choices of the Sobolev spaces involved and this will be discussed in Section~\ref{sec:TA} where \ref{H:approx} is shown for the averaged Taylor expansion operator.

\subsection{Local error}

Usually local error bounds are stated using the \emph{support extension}
$$
\begin{aligned}&\norm{ f-\aop f}_{p, \omega}\le C\inf_{g\in\poly}\norm{f-g}_{p, \extension\omega},&&& \extension\omega:=\bigcup_{\phi\in\actives_\omega}\support\lambda_\phi.
\end{aligned}
$$
The drawback of this approach is that the polynomial approximation properties depends on the shape of $\extension\omega$.
Common assumption for polynomial approximation are that the domain is star-shaped with respect to a subset of positive measure \cite{dupont_scott}\cite{duran}\cite{verfurth} or bounded by graphs of regular functions \cite{reif}.
In \cite{MR3627466} local estimates for the approximation error are obtained without shape constraints on $\extension\omega$ by using the convex hull of $\extension\omega$. Unfortunately this technique cannot be extended to global estimates because it is difficult to bound the number of overlapping hulls in terms of elementary properties of the mesh.

Here we avoid the support extension by using the form
\begin{equation}\label{eq:approx-form}
\norm{f-\aop f}_{p, \omega}\le C\inf_{g\in\poly} \sum_{\phi\in\actives_\omega}\norm{f-g}_{p, \esupp_\phi}.
\end{equation}
This reduces the complexity of the shape on which we use polynomial approximation and allow us to derive global estimates and to reduce the mesh regularity assumptions.

The local error bounds are based on the following assumptions
\begin{hlist}
\hitem{polynomial}{\poly} $\aop g = g $ for all $g\in\poly_{\vd}$, \emph{polynomial reproduction};
\hitem{lambda}{\lambda} $\norm{\lambda_\phi}_{*p}\norm{\phi}_{p}\le\Clambda$, for all $\phi\in\Phi$, \emph{functional continuity};
\hitem{phi}{\phi} $\partial^{\vsigma}\phi\in L^p(\Omega)$, for all $\phi\in\Phi$, and  $\norm{\partial^{{\vsigma}}\phi}_{p, \omega}\le\Cphi\hp^{-{\vsigma}}\frac{\measure{\omega}^{\frac1p}}{\measure{\support\phi}^{\frac1p}}\norm{\phi}_{p}$, \emph{generators' regularity};
\hitem{esupp}{{\esupp}} $\ha\le\Cesupp\hp $, and $\ha^{\vec1}\le\Cesupp^n\mu(\support\phi)$ for all $\phi\in\Phi$, \emph{locality of $\aop$};
\hitem{approx}{\Pi} for each $\omega\in\mesh$ there is an approximation operator $\paop_\omega: L^p(\omega)\to\poly_{\vd}$ such that for  ${\vbeta}\in\{\vec0, \vsigma\}$, $\eta\in\{\omega\}\cup\{\esupp_\phi:\phi\in\actives_\omega\}$ it holds
$$\bnorm{\partial^{\vbeta}(f-\paop_\omega f)}_{p, \eta}\le\Capprox\frac{\ho^{{\vgamma}}}{\he^{{\vgamma}}}\sum_{\vk\in K_{\vbeta}}\he^{\vk-{\vbeta}+\vpq}\bnorm{\partial^{\vk} f}_{q, \eta},\qquad \vpq={\vec1}/p-{\vec1}/q.$$
\end{hlist}
Note that $\Cesupp\ge 1$, and for simplicity we assume that also $\Clambda,\Cphi \ge 1$.

The abstract upper bound for polynomial approximation in \ref{H:approx} depends on the index sets $K_{\vec0},\,K_{\vsigma}\subset\N^n$ and the parameter $\vgamma\in\R^n$ that influences the required mesh assumptions.
In Section~\ref{sec:estimates} we describe operators $\paop_\omega$ leading to error bounds with Sobolev and reduced seminorms.

\begin{remark}
If $\support\lambda_\phi\subseteq\support\phi=\bbox(\esupp_\phi)$ then \ref{H:esupp} is satisfied with $\Cesupp=1$. This will be the case for the proposed operators for TPS, LR and AST splines.
The control on the support measure in \ref{H:esupp} is needed if $\support\phi$ or $\esupp_\phi$ is not a box. This happens for THB and non-convex domains.
\end{remark}

\begin{theorem}[Local error]
\label{thm:local} For \mbox{$1\le p, q\le\infty$} and assuming \ref{H:polynomial},\dots,\ref{H:approx}  we have
\begin{equation}\label{eq:local}
\begin{aligned}
&\bnorm{\partial^{\vsigma}(f-\aop f)}_{p, \omega}\le\Clambda\Cphi\Cesupp^{\abs{\vsigma}+\frac{n}{p}}\Capprox
\\&\quad
\bigg[\sum_{\vk\in K_{\vsigma}}\ho^{\vk-\vsigma+\vpq}\bnorm{\partial^{\vk} f}_{q, \omega}+\sum_{\phi\in\actives_\omega}
\frac{\ho^{{\vgamma}+\vec1/p}}{\ha^{{\vgamma}+\vec1/p}}\sum_{\vk\in K_{\vec0}}\ha^{\vk-\vsigma+\vpq}\bnorm{\partial^{\vk} f}_{q, \esupp_\phi}\bigg ].
\end{aligned}
\end{equation}

Moreover, if $\vsigma=\vec0$ and $p=\infty$ the same inequality holds under the reduced assumptions  \ref{H:polynomial}, \ref{H:lambda},\ref{H:phi} and \ref{H:approx}.
\end{theorem}

\begin{proof}
The reproduction of polynomial property \ref{H:polynomial} and linearity of $\aop$ imply that for all $g\in \poly_{\vd}$
\begin{align*}
\bnorm{\partial^{\vsigma} (f-\aop f)}_{p, \omega}&=\bnorm{\partial^{\vsigma}(f-g)-\partial^{\vsigma}\aop (f-g)}_{p, \omega}\\
&=\bnorm{\partial^{\vsigma}(f-g)-\sum_{\phi\in\actives_\omega}\lambda_\phi(f-g)\partial^{\vsigma}\phi}_{p, \omega}\\
&\le\norm{\partial^{{\vsigma}}(f-g)}_{p, \omega}+\sum_{\phi\in\actives_\omega}\norm{\partial^{{\vsigma}}\phi}_{p, \omega}\abs{\lambda_\phi(f-g)}.
\end{align*}
By \ref{H:esupp} we obtain
\begin{equation}\label{eq:proof-local-0}
\hp^{-{\vsigma}}\le\Cesupp^{\abs{\vsigma}}\ha^{-{\vsigma}},\quad\measure{\omega}^{1/p}=\ho^{\vec1/p},\quad\measure{\support\phi}^{-1/p}\le\Cesupp^{n/p}\ha^{-\vec1/p}.
\end{equation}
Using this, \ref{H:phi}, \ref{H:lambda} and \ref{H:approx} with $g=\paop_\omega f$ we have
\begin{equation}\label{eq:proof-local-1}
\begin{aligned}
&\norm{\partial^{{\vsigma}}\phi}_{p, \omega}\abs{\lambda_\phi(f-g)}
\le\Cphi\ha^{-\vsigma}\frac{\measure{w}^{1/p}}{\measure{\support\phi}^{1/p}} \norm{\phi}_p \norm{\lambda_\phi}_{*p}\norm{f-g}_{p,\support\lambda_\phi}
\\&\qquad\le \Cphi\Cesupp^{\abs{\vsigma}}\ha^{-\vsigma}\ho^{\vec1/p}\Cesupp^{n/p}\ha^{-\vec1/p}\Clambda\Capprox \ho^{\vgamma}\ha^{-\vgamma}\sum_{\vk\in K_{\vec0}}\ha^{\vk+\vpq}\bnorm{\partial^{\vk}f}_{q,\esupp_\phi}.
\end{aligned}
\end{equation}
Summing over the $\phi\in\actives_\omega$ and simplifying we obtain the second term in \eqref{eq:local}.
Using \ref{H:approx} with $\eta=\omega$ and $\vbeta=\vsigma$ we have
\begin{equation}
\label{eq:proof-local-2}
\bnorm{\partial^{\vsigma}(f-\paop_\omega f)}_{p, \omega}\le\Capprox\sum_{\vk\in K_{\vsigma}}\ho^{\vk-\vsigma+\vpq}\bnorm{\partial^{\vk} f}_{q, \omega}
\end{equation}
and  \eqref{eq:local} follows since $\Clambda\Cphi\Cesupp^{\abs{\vsigma}+\frac{n}{p}}\ge 1$.
\qed\end{proof}

\subsection{Global error \texorpdfstring{$p<\infty$}{p<infty}}

The global error bound is obtained by summing the local bounds from Theorem~\ref{thm:local}.
The procedure uses the following additional assumptions
\begin{hlist}
\hitem{duals}{\actived} $\cardinality{\actived_\omega}\le\Cduals$, for all $\omega\in\mesh$, \emph{bound on $\esupp_\phi$ overlaps};
\hitem{mesh}{\pmesh} $\Gamma_\phi\le\Cmesh$
for all $\phi\in\Phi$, \emph{ mesh regularity}, where
$$
\Gamma_\phi:=\begin{cases}
\Big(\sum_{\omega\in\contained_\phi}\frac{\ho^{{\vgamma} p+\vec1}}{\ha^{{\vgamma} p+\vec1}}\Big )^{\frac1p} & p\ne\infty\\
\max_{\omega\in\contained_\phi}\Big\{\frac{\ho^{{\vgamma}}}{\ha^{{\vgamma}}}\Big\} & p=\infty,
\end{cases}
$$
and ${\vgamma}$ was introduced in \ref{H:approx};
\hitem{isotropicO}{\pshapeO} $\vmax\ho\le\CisotropicO\, \vmin\ho $ for all $\omega\in\mesh$, \emph{element shape regularity};
\hitem{isotropicA}{\pshapeA} $\vmax\ha\le\CisotropicA\, \vmin\ha $ for all $\phi\in\Phi$,  \emph{shape regularity of $\esupp_\phi$}.
\end{hlist}
Note that $\Cduals,\CisotropicO,\CisotropicA \ge 1$, and for simplicity we assume that also $\Cmesh \ge 1$.

\ref{H:mesh} is a mesh regularity property that mirrors the result of the computations.
It is implied by the usual assumptions for non-tensor-product spaces, but it allows for weaker mesh regularity if $\vgamma\ge -\vec1/p$.
In particular tiny and not shape regular elements are allowed as long as each extended support is shape regular and contains a bounded number of elements, see Subsection~\ref{sec:mesh-assumptions}.
Assumptions \ref{H:isotropicO} and \ref{H:isotropicA} will be used only if $\vsigma>0$ or if $p>q$.

\begin{remark}\label{rem:actives-extended}
As $\actived_\omega\supseteq\actives_\omega$, we have $\cardinality{\actives_\omega}\le\cardinality{\actived_\omega}\le\Cduals$.
Some estimates can be made slightly sharper by using an upper bound for $\cardinality{\actives_\omega}$. 
\end{remark}

\begin{theorem}[Global error]
\label{thm:global}For $1\le p<\infty$, $1\le q\le\infty$ and 
assuming \ref{H:polynomial},\dots,\ref{H:mesh} we have
\begin{equation}\label{eq:global}
\bnorm{\partial^{\vsigma}\!(f-\aop f)}_{p}\le \Ctot \measure{\Omega}^{\pq_+}\Big( \sum_{\vk\in K_{\vsigma}}\!\!\bnorm{\hof^{\vk-\vsigma+\vpq_-}\partial^{\vk} f}_{q}\!\! + \sum_{\vk\in K_{\vec0}}\!\!\bnorm{\rho_{\vk, {\vsigma}}\partial^{\vk} f}_{q}
\Big),
\end{equation}
where we recall that $\vpq:=\vec1/p-\vec1/q$, $\hof$ is defined in \eqref{eq:def-ho}, $\rho_{\vk, {\vsigma}} :\Omega\to\R$ is the piecewise constant function defined by $\rho_{\vk, {\vsigma}}|_\omega :=\maxv_{\phi\in\actived_\omega}\{\ha^{\vk-\vsigma+\vpq_-}\}$ and
\begin{equation}\label{eq:global-constant}
\begin{aligned}
\Ctot &:= 2^{1-\frac1p} \cardinality{(K_{\vsigma}\cup K_{\vec0})}^{1-1/p} \Clambda\Cphi  \Cesupp^{\abs{\vsigma}+\abs{\vpq_+}+\frac np} \Capprox \Cduals^{1-\pq_-}\Cmesh.
\end{aligned}
\end{equation}
\end{theorem}
Before giving the proof we state two Corollaries.

\begin{corollary}\label{thm:global-ani}
For $1\le p\le q\le\infty$, $p\ne\infty$ and assuming \ref{H:polynomial},\dots,\ref{H:mesh} we have
\begin{equation}\label{eq:global-ani}
\norm{f-\aop f}_{p}\le \Ctot\,\measure{\Omega}^{\pq_+} \sum_{\vk\in K_{\vec0}}\bnorm{\hm^{\vk}\partial^{\vk} f}_{q}
\end{equation}
where $\Ctot:=4\cardinality{K_{\vec0}}^{1-\frac1p} \Clambda\Cphi\Cesupp^{n(\frac2p-\frac1q)}\Capprox\Cduals\Cmesh$.
\end{corollary}

\begin{proof}
Since $\hof|_\omega=\ho\le \maxv_{\phi\in\actived_\omega} \ha=\hm|_\omega$ and  $K_{\vec0}\ge \vec0$ we have
$\hof^{\vk}\le\hm^{\vk}.$
Moreover, $\vpq_-=\vec0$ because $p\le q$, so that
$$
\rho_{\vk, \vec0}|_\omega=\max_{\phi\in\actived_\omega}\prod_{i=1}^n\hai i^{ k_i} \le \prod_{i=1}^n \max_{\phi\in\actived_\omega}\hai i^{ k_i} =\hm^{\vk} |_\omega.
$$
Inserting  the above in Theorem~\ref{thm:global} gives the result.
\qed\end{proof}

\begin{corollary}
\label{thm:global-iso}
For $1\le p, q\le\infty$, $p\ne\infty$ and assuming \ref{H:polynomial},\dots,\ref{H:isotropicA}, $ K_{\vsigma}\ge\vsigma$ and $\abs{\vk}\ge \abs{\vsigma} + \abs{\vpq_-}$ for $\vk\in K_{\vec0}\cup K_{\vsigma}$, we have
\begin{equation}\label{eq:global-iso}\begin{aligned}
\bnorm{\partial^{\vsigma}(f-\aop f)}_{p} \le &\Ctot (\CisotropicO^{\abs{\vpq_-}}\! +\CisotropicA^ {\abs{\vsigma}+\abs{\vpq_-}}) \\&\measure{\Omega}^{\pq_+}\!\!\!\sum_{\vk\in K_{\vsigma}\cup K_{\vec0}}\bnorm{(\vmax\hm)^{\abs{\vk}-\abs{\vsigma}-\abs{\vpq_-}}\partial^{\vk} f}_{q},
\end{aligned}
\end{equation}
where $\Ctot$ is given in Theorem~\ref{thm:global}.
Moreover, if $p\le q$ then \ref{H:isotropicO} is not required. If $p\le q$ and $\vsigma=\vec0$ then \ref{H:isotropicO} and \ref{H:isotropicA} are not required.
\end{corollary}

\begin{proof}
With $\valpha=\vk-\vsigma+\vpq_-$ and using \eqref{eq:power-identity} and \ref{H:isotropicO}  we have
\begin{equation}\label{eq:global-iso-1}
\begin{aligned}
\hof^{\valpha}|_\omega&\le (\max \ho)^{\abs{\valpha_+}}\, (\min \ho)^{-\abs{\valpha_-}}\le \CisotropicO^{\abs{\valpha_-}}(\max \ho)^{\abs{\valpha_+}- \abs{\valpha_-}}.
\end{aligned}
\end{equation}
Since $\vk\ge\vsigma$ we have $\abs{\vpq_-}\ge\abs{\valpha_-}$ and using $\CisotropicO\ge1$ and $\vmax\ho\le\vmax\hm|_\omega$  we obtain \begin{equation}
\label{eq:ho-isotropic}\hof^{\vk-\vsigma+\vpq_-}\le\CisotropicO^{\abs{\vpq_-}} (\vmax\hm)^{\abs{\vk}-\abs{\vsigma}-\abs{\vpq_-}}.
\end{equation}
Note that for $p\le q$ and $\vk\in K_{\vsigma}$ then $\valpha\ge\vec0$ and \ref{H:isotropicO} is not required.
Similarly using \ref{H:isotropicA} and $\abs{\valpha_-}\le\abs{\vsigma}+\abs{ \vpq_-} $ we obtain
\begin{equation}
\label{eq:ha-isotropic}
\rho_{\vk, {\vsigma}}\le\CisotropicA^{\abs{\vsigma}+\abs{\vpq_-}} (\vmax\hm)^{\abs{\vk}-\abs{\vsigma}-\abs{\vpq_-}}.
\end{equation}
Moreover, for $p\le q$ and $\vsigma=0$ we have $K_{\vec0}=K_{\vsigma}\ge\vec0 $  and \ref{H:isotropicA} is not required.
Inserting \eqref{eq:ho-isotropic} and \eqref{eq:ha-isotropic} in Theorem~\ref{thm:global} gives the result.
\qed\end{proof}

\begin{proof}[Theorem~\ref{thm:global}]
\label{sec:global-proof}
From $\norm{\partial^{\vsigma}(f-\aop f)}_{p}= \Big( \sum_{\omega\in\mesh} \norm{\partial^{\vsigma}(f-\aop f)}_{p,\omega}^p\Big)^{1/p}$, \eqref{eq:local} and \eqref{eq:sums} with $e=p\ge 1$ and $m=\cardinality{K_{\vsigma}}+\cardinality{\actives_\omega}\cardinality{K_{\vec0}}\le 2\cardinality{(K_{\vsigma}\cup K_{\vec0})} \Cduals$ and again \eqref{eq:sums}  with $e=1/p\le 1$ we get
\begin{equation*}
\begin{aligned}
&\norm{\partial^{\vsigma}(f-\aop f)}_{p} \le 2^{1-\frac1p} \cardinality{(K_{\vsigma}\cup K_{\vec0})}^{1-\frac1p} \Clambda\Cphi\Cesupp^{\abs{\vsigma}+n/p}\Capprox \Cduals^{1-\frac1p}\\
&\qquad
\bigg[\sum_{\vk\in K_{\vsigma}} \Big( \sum_{\omega\in\mesh}\ho^{p(\vk-\vsigma+\vpq)}\bnorm{\partial^{\vk} f}_{q, \omega}^p \Big)^{\frac1p}\\
&\qquad\qquad +\sum_{\vk\in K_{\vec0}}\Big(\sum_{\omega\in\mesh}\sum_{\phi\in\actives_\omega}
\frac{\ho^{{\vgamma}p+\vec1}}{\ha^{{\vgamma}p+\vec1}}\ha^{p(\vk-\vsigma+\vpq)}\bnorm{\partial^{\vk} f}_{q, \esupp_\phi}^p \Big)^{\frac1p}\bigg ].
\end{aligned}
\end{equation*}
By changing the summation order we have
$$
\begin{aligned}
\sum_{\omega\in\mesh}\sum_{\phi\in\actives_\omega}
\frac{\ho^{{\vgamma}p+\vec1}}{\ha^{{\vgamma}p+\vec1}}\ha^{p(\vk-\vsigma+\vpq)}\bnorm{\partial^{\vk} f}_{q, \esupp_\phi}^p&=\sum_{\phi\in\Phi}\sum_{\omega\in\contained_\phi}
\frac{\ho^{{\vgamma}p+\vec1}}{\ha^{{\vgamma}p+\vec1}}\ha^{p(\vk-\vsigma+\vpq)}\bnorm{\partial^{\vk} f}_{q, \esupp_\phi}^p
\end{aligned}
$$
that using \ref{H:mesh} leads to
\begin{equation}
\label{eq:global-structure}
\begin{aligned}
&\norm{\partial^{\vsigma}(f-\aop f)}_{p} \le 2^{1-\frac1p} \cardinality{(K_{\vsigma}\cup K_{\vec0})}^{1-\frac1p} \Clambda\Cphi\Cesupp^{\abs{\vsigma}+\frac np} \Capprox\Cduals^{1-\frac1p} \Cmesh
\\&\quad \bigg[\sum_{\vk\in K_{\vsigma}}\Big(\sum_{\omega\in\mesh}\ho^{p(\vk-\vsigma+\vpq)}\bnorm{\partial^{\vk} f}_{q, \omega}^p\Big)^{\frac1p}+ \sum_{\vk\in K_{\vec0}}\Big(\sum_{\phi\in\Phi}\ha^{p(\vk-\vsigma+\vpq)}\bnorm{\partial^{\vk} f}_{q, \esupp_\phi}^p\Big)^{\frac1p}\bigg].
\end{aligned}
\end{equation}
For $p \le q$, H\"older inequality \eqref{eq:holder} with $a_i=\ho^{p\vpq}$ and $e=(p\pq)^{-1}\ge 1$ implies
\begin{equation}
\label{eq:bracket-1-plq}
\begin{aligned}
\sum_{\omega\in\mesh}\ho^{p(\vk-\vsigma+\vpq)}\bnorm{\partial^{\vk} f}_{q, \omega}^p &\le\Big (\sum_{\omega\in\mesh}\ho^{\vec1}\Big )^{p\pq}\Big (\sum_{\omega\in\mesh}\ho^{q(\vk-\vsigma)}\bnorm{\partial^{\vk} f}_{q, \omega}^q\Big )^{\frac pq}\\
&\le\measure{\Omega}^{p\pq}\bnorm{\hof^{\vk-\vsigma}\partial^{\vk} f}_{q}^{p}.
\end{aligned}
\end{equation}
Similarly we get 
\begin{equation}\label{eq:bracket-2-plq}
\begin{aligned}
\sum_{\phi\in\Phi}\ha^{p(\vk-\vsigma+\vpq)}\bnorm{\partial^{\vk} f}_{q, \esupp_\phi}^p &\le\Big (\sum_{\phi\in\Phi}\ha^{\vec1}\Big )^{p\pq}\Big (\sum_{\phi\in\Phi}\ha^{q(\vk-\vsigma)}\bnorm{\partial^{\vk} f}_{q, \esupp_\phi}^q\Big )^{\frac pq}
\\&\le\Cduals\Cesupp^{np\pq}\measure{\Omega}^{p\pq}\bnorm{\rho_{\vk, {\vsigma}}\partial^{\vk} f}_{q}^{p},
\end{aligned}
\end{equation}
where we used  \ref{H:esupp} and \ref{H:duals} to obtain
$$\sum_{\phi\in\Phi}\ha^{\vec1}\le\Cesupp^n\sum_{\phi\in\Phi}\measure{\support\phi}\le\Cesupp^n\Cduals\measure{\Omega},$$
and the characteristic functions $\ind_{\esupp_\phi}$ to obtain
\begin{align*}
\sum_{\phi\in\Phi}\ha^{q(\vk-\vsigma)}\bnorm{\partial^{\vk} f}_{q, \esupp_\phi}^q &=\int_\Omega\Big(\sum_{\phi\in\Phi}\ind_{\esupp_\phi}(\vx)\ha^{q(\vk-\vsigma)}\Big )\babs{\partial^{\vk} f (\vx)}^q\, d\vx\\
&\le\int_\Omega\Cduals\rho_{\vk, {\vsigma}}^q\babs{\partial^{\vk} f(\vx)}^q\, d\vx =\Cduals \bnorm{\rho_{\vk, {\vsigma}}\partial^{\vk} f}_{q}^{q}.
\end{align*}
Inserting \eqref{eq:bracket-1-plq} and \eqref{eq:bracket-2-plq} in \eqref{eq:global-structure} gives \eqref{eq:global} for $p\le q$.

If $p> q$ then using \eqref{eq:sums} with $e=p/q\ge1$  we have
\begin{equation}
\label{eq:bracket-1-pgq}
\begin{aligned}
\sum_{\omega\in\mesh}\ho^{p(\vk-\vsigma+\vpq)}\bnorm{\partial^{\vk} f}_{q, \omega}^p &\le
(\sum_{\omega\in\mesh}\ho^{q(\vk-\vsigma+\vpq)}\bnorm{\partial^{\vk} f}_{q, \omega}^q)^{\frac{p}{q}}\\
&=\bnorm{\hof^{\vk-\vsigma+\vpq}\partial^{\vk} f}_{q}^p.
\end{aligned}
\end{equation}
Similarly
\begin{equation}
\label{eq:bracket-2-pgq}
\begin{aligned}
&\sum_{\phi\in\actives_\omega}
\ha^{p(\vk-\vsigma+\vpq)}\bnorm{\partial^{\vk} f}_{q, \esupp_\phi}^p
\le\Big(\sum_{\phi\in\actives_\omega}\ha^{q(\vk-\vsigma+\vpq)}\bnorm{\partial^{\vk} f}_{q, \esupp_\phi}^q\Big)^{\frac{p}{q}}\\
&\qquad\qquad=\Big(\sum_{\phi\in\Phi}\int_\Omega\ind_{\esupp_\phi}(\vx)\ha^{q(\vk-\vsigma+\pq)}\babs{\partial^{\vk} f(\vx)}^q\, d\vx\Big)^{\frac{p}{q}}\\
&\qquad\qquad\le\Cduals^{\frac pq}\bnorm{\rho_{\vk, {\vsigma}}\partial^{\vk} f}_{q}^{p}.
\end{aligned}
\end{equation}
Inserting \eqref{eq:bracket-1-pgq} and \eqref{eq:bracket-2-pgq} in \eqref{eq:global-structure} gives
\eqref{eq:global} for $p\ge q$.
\qed\end{proof}

\subsection{Global error \texorpdfstring{$p=\infty$}{p=infty}}

For $p=\infty$, \ref{H:duals} can be replaced with a bound on $\cardinality{\actives_\omega}$, see Remark~\ref{rem:actives-extended}.
Moreover, we have $\vpq=\vpq_-=-\vec1/q$.

\begin{theorem}
\label{thm:global-inf}
Assuming \ref{H:polynomial},\dots,\ref{H:mesh}, we have
\begin{align}
\label{eq:global-inf-ani}&\norm{f-\aop f}_{\infty}\le\Ctot\max_{\phi\in\Phi}\Big\{\sum_{\vk\in K_{\vec0}}\bnorm{\ha^{\vk}\partial^{\vk} f}_{\infty, \esupp_\phi}\Big\},
\end{align}
where $\Ctot :=2\Clambda\Cphi\Capprox\Cesupp^{\abs{\vsigma}}\Cactives\Cmesh.$
Assuming in addition \ref{H:isotropicO}, \ref{H:isotropicA},  $ K_{\vsigma}\ge\vsigma$ and $\abs{\vk}\ge \abs{\vsigma} + \abs{\vpq_-}$ for $\vk\in K_{\vec0}\cup K_{\vsigma}$ we have with the same $\Ctot$
\begin{align}
\label{eq:global-inf-iso}
\begin{aligned}
	\norm{\partial^{\vsigma}(f-\aop f)}_{\infty}\le&\Ctot (\CisotropicO^{\abs{\vpq_-}}+\CisotropicA^ {\abs{\vsigma}+\abs{\vpq_-}}\big ) 
\\&\max_{\phi\in\Phi}\Big\{\sum_{\vk\in K_{\vsigma}\cup K_{\vec0}}\bnorm{(\vmax\ha)^{\abs{\vk}-\abs{\vsigma}-\abs{\vpq_-}}\partial^{\vk} f}_{q, \esupp_\phi}\Big\}.\end{aligned}\end{align}
Moreover, if $q=\infty$ then \eqref{eq:global-inf-iso} does not require \ref{H:isotropicO}.
\end{theorem}

\begin{proof}
By \ref{H:mesh}, \eqref{eq:local} and $ \cardinality{\actives_\omega}\le \Cduals$ we have
\begin{equation}\label{eq:global-inf-1}
\begin{aligned}
 &\norm{\partial^{\vsigma}(f-\aop f)}_{\infty, \omega}\le
\Clambda\Cphi\Cesupp^{\abs{\vsigma}}\Capprox\Cduals\Cmesh
\\&\qquad\qquad\bigg[\sum_{\vk\in K_{\vsigma}}\ho^{\vk-\vsigma+\vpq}\bnorm{\partial^{\vk} f}_{q, \omega}+\max_{\phi\in\actives_\omega}
\sum_{\vk\in K_{\vec0}}\ha^{\vk-\vsigma+\vpq}\bnorm{\partial^{\vk} f}_{q, \esupp_\phi}\bigg ].
\end{aligned}
\end{equation}
If $q=\infty$ and $\vsigma=\vec0$, then $\vpq=\vec0$ and $\ho^{\vk}\le \ha^{\vk}$, so that the first term in brackets is bounded by the second and we get \eqref{eq:global-inf-ani}.

If $q<\infty$ or $\vsigma\ne\vec0$ we use \eqref{eq:power-identity},  \ref{H:isotropicO} and \ref{H:isotropicA} to obtain
\begin{align*}
\ho^{\vk-\vsigma+\vpq}&\le\CisotropicO^{\abs{\vpq}}(\vmax\ha)^{\abs{\vk}-\abs{\vsigma}-\abs{\vpq}},&&& \vk\ge\vsigma, \\
\ha^{\vk-\vsigma+\vpq}&\le\CisotropicA^{\abs{\vsigma}+\abs{\vpq}}(\vmax\ha)^{\abs{\vk}-\abs{\vsigma}-\abs{\vpq}},
\end{align*}
and inserting these in  \eqref{eq:global-inf-1} gives \eqref{eq:global-inf-iso}.
If $q=\infty$ then \ref{H:isotropicO} is not required.
\qed\end{proof}

\subsection{Mesh assumptions}
\label{sec:mesh-assumptions}

The assumption \ref{H:mesh} can be replaced by one of the following
\begin{enumerate}
\hsubitem{mesh-num}{\#}{mesh} $\cardinality{  \contained_\phi}\le\Cnum$, \, for all $\phi\in\Phi$, \emph{$\#$elements in $\support\phi$};
\hsubitem{mesh-length}{\mathrm m}{mesh} $\ha\le\Clength\ho$, \, for all $\omega\in\mesh$ and $\phi\in\actives_\omega$, \emph{mesh quasi-uniformity}.
\end{enumerate}
We show in Proposition~\ref{thm:mesh-assumption} that \ref{H:mesh-num} implies \ref{H:mesh} for ${\vgamma}\ge-{\vec1}/p$ and that \ref{H:mesh-length} implies \ref{H:mesh} for all values of $\vgamma$.

\begin{proposition}
\label{thm:mesh-assumption}
The conditions listed in the following table imply \ref{H:mesh} with the corresponding constants $\Cmesh$.
\global\def\uglyCmesh{\Clength^{\abs{(\vgamma+\vec1/p)_-}+\frac np}}
\begin{center}
\begin{tabular}{rl|l}
\multicolumn{2}{c|}{conditions}&$\Cmesh$\\
\hline
 $\vec0\le {\vgamma}$\phantom{,} && $1$\\
 $-\vec1/p\le {\vgamma}$, & \ref{H:mesh-num} &$\Cnum^{-\vmin\vgamma}$\\
 $\vgamma\in\R^n$, &\ref{H:mesh-length} &$\uglyCmesh$
 \end{tabular}
\end{center}
\end{proposition}

\begin{proof}
First we note that for all $\valpha\ge 1$ we have
$$
\sum_{\omega\in\contained_\phi}
\frac{\ho^{\valpha}}{\ha^{\valpha}}=\sum_{\omega\in\contained_\phi}\prod_{i=1}^n\Big(\frac{\hoi i}{\hai i}\Big)^{\alpha_i}\!\!\le \sum_{\omega\subseteq\esupp_\phi}\frac{\ho^{\vec1}}{\ha^{\vec1}} = \frac{\measure{\support\phi}}{\measure{\bbox(\esupp_\phi)}} \le 1.
$$	
This gives the first case as  $\valpha=\vgamma p +\vec1\ge\vec1$.
For $\vgamma p\ge-\vec1/p$, we use \eqref{eq:sums} with $ e=\vmin{\vgamma} p+1$, \ref{H:mesh-num}, and with $\valpha=(\vgamma p+\vec1)/(\vmin{\vgamma} p+1)\ge \vec1$ we obtain
\begin{align*}
\sum_{\omega\in\contained_\phi} \frac{\ho^{\vgamma p+1}}{\ha^{\vgamma p+1}}
&\le \cardinality{\contained_\phi}^{-\vmin{\vgamma} p} \Big(\sum_{\omega\in\contained_\phi} 
\frac{\ho^{\valpha}}{\ha^{\valpha}}\Big)^{\vmin{\vgamma} p+1}\le \Cnum^{-\vmin{\vgamma} p}.
\end{align*}
This shows the second case.
For the last case, if \ref{H:mesh-length} holds then
$$
\measure{\esupp_\phi}\ge \sum_{\omega\in\contained_\phi} \measure{\omega}\ge \cardinality{\contained_\phi}\min_{\omega\in\contained_\phi}\measure{\omega}\ge
\cardinality{\contained_\phi} \Clength^{-n}\measure{\esupp_\phi}.
$$
Therefore \ref{H:mesh-num} holds with $\Cnum=\Clength^n$. 
Using \ref{H:mesh-length} again and with $\valpha=\vgamma p+\vec1$ we have
$$
\sum_{\omega\in\contained_\phi}\frac{\ho^{\valpha}}{\ha^{\valpha}} \le\cardinality{\contained_\phi} \max_{\omega\in\contained_\phi}\frac{\ho^{\valpha}}{\ha^{\valpha}} \le \Clength^n \, \Clength^{\abs{\valpha_-}}.
$$
 \qed\end{proof}

An interesting observation is that \ref{H:mesh-num} and \ref{H:mesh-length} influence the relation between \ref{H:isotropicO} and \ref{H:isotropicA}. Assuming \ref{H:mesh-num} we have that \ref{H:isotropicO} implies \ref{H:isotropicA}. Assuming \ref{H:mesh-length} the two shape regularity assumptions are equivalent. This is proved in Proposition~\ref{thm:shape-regularity}.

\begin{proposition}\label{thm:shape-regularity}
We have
\begin{align*}
\text{\ref{H:mesh-num}}\ &\thus\ (\text{ \ref{H:isotropicO} \thus\ \ref{H:isotropicA} }),\\
\text{\ref{H:mesh-length}}\ &\thus\ (\text{ \ref{H:isotropicO} \iff\ \ref{H:isotropicA} }) \text{ and }\Clength^{-1}\CisotropicA\le\CisotropicO\le\Clength\CisotropicA.
\end{align*}
\end{proposition}

\begin{proof}
Let $\omega\subseteq \esupp_\phi$ be the element having the longest edge $\hoi i$. Then by \ref{H:mesh-num} and \ref{H:isotropicO}
$$
\vmax\ha\le\Cnum\vmax\ho\le\Cnum\CisotropicO\vmin\ho\le\Cnum\CisotropicO\vmin\ha.
$$
Let $i,j\in\{1,\dots,n\}$ be such that  $\min\ho=\hoi i$ and $\min\ha=\hai j$. Then using \ref{H:mesh-length} and \ref{H:isotropicA} we have $$
\begin{aligned}
	&\vmax\ho\le\vmax\ha\le\Cisotropic\vmin\ha
	\\&\qquad = \Cisotropic \hai j\le \Cisotropic \hai i \le \Cisotropic\Clength \hoi i =\Cisotropic\Clength\vmin\ho
	\end{aligned}
	$$
 and \ref{H:isotropicO} follows.
  Similarly with the same $i,j$ and using \ref{H:mesh-length} and \ref{H:isotropicO} we have 
$$
\begin{aligned}
&\vmax\ha\le\Clength \vmax\ho\le\CisotropicO\Clength\vmin\ho
\\ &\qquad= \CisotropicO\Clength\hoi i \le \CisotropicO\Clength\hoi j \le  \CisotropicO\Clength \hai j = \CisotropicO\Clength\vmin\ha
\end{aligned}
$$
and \ref{H:isotropicA} follows.
\qed\end{proof}

Fig.~\ref{figure:isotropic-example} shows a family of LR spaces \cite{MR3146870}\cite{MR3019748} on which \ref{H:mesh-num} and \ref{H:isotropicO} hold uniformly, but $\CisotropicO$ is not bounded.
The element $\omega$ (\tikz\fill[\RF] (0, 0) rectangle (.3, .3);) has size $\ho=(1, \varepsilon)$ and $\support\phi=\esupp_\phi$ (\tikz\fill[pattern=custom north east lines, hatchcolor=black, hatchspread=2\hatchspread] (0, 0) rectangle (.3, .3);) has size $\ha=(3, 1)$.
It can be shown that \ref{H:mesh-num} holds with $\Cnum=13$, \ref{H:isotropicO} holds with $\CisotropicO=\varepsilon^{-1}$ and \ref{H:isotropicA} holds with $\CisotropicA=3$.
By letting $\varepsilon\to0$ we observe that no bound for $\CisotropicO$ can be a function of $\Cnum$ and $\CisotropicA$.

\begin{figure}\begin{center}
\begin{tikzpicture}
\fill[\FF](0, 2) rectangle (5, 5);
\fill[\RF] (3, 3.3) rectangle (4, 3.7) node[midway] (A){};
\fill[pattern=custom north east lines, hatchcolor=black, hatchspread=2\hatchspread] (1, 3) rectangle (4, 4) ;
\draw (0, 2) grid (5, 5);
\foreach\x in {3.3, 3.7}\draw (1, \x)--(4, \x);
\draw[decorate, decoration={brace, raise=+5pt, amplitude=.5ex}](1, 3.3)--(1, 3.7) node[midway, left=9pt] {$\varepsilon$};
\draw (A) node {$\omega$};
\draw (2.5, 4) node[above]{$\esupp_\phi$};
\end{tikzpicture}\end{center}
\caption{A family of LR spaces on which \ref{H:mesh-num} and \ref{H:isotropicO} hold uniformly, but $\CisotropicO$ is not bounded.}
\label{figure:isotropic-example}
\end{figure}

\section{Polynomial approximation}\label{sec:TA}

This section describes collections of operators $\Pi_\omega$ that satisfy \ref{H:approx}.
The construction is based on averaged Taylor expansion operators.
To an index set $A\subseteq\N^n$ and a weight function $\psi\in L^1(\Omega)$ such that $\int\psi=1$ we associate the operator $T_{A, \psi}:\smooth^\infty(\Omega)\to\poly_A$ defined by
\begin{equation}
\label{eq:TA-general}
T_{A, \psi} f (\vx) :=\sum_{{\valpha}\in A}\int_{\Omega}\psi(\vy)\frac{(\vx-\vy)^{\valpha}}{{\valpha}!}\partial^{\valpha} f(\vy)\, d\vy.
\end{equation}
It is required that $\poly_A$ is translation invariant
\begin{equation}
\label{eq:TA-a-translation}\forall \vy\in\R^n, \quad\poly_A :=\SPAN\{\vx^{\valpha}:{\valpha}\in A\}=\SPAN\{(\vx-\vy)^{\valpha}:{\valpha}\in A\}.
\end{equation}
Different choices are possible. 
The error bounds in terms of Sobolev and reduced seminorms use $A=\{\valpha\in\N^n:\abs{\valpha}\le d\}$ and $A=\{\valpha\in\N^n: {\valpha}\le \vd\}$, respectively. 

The operators $T_{A, \psi}$ are defined on $\smooth^\infty(\Omega)$ and are uniquely extended to a Sobolev space $W$, provided $\smooth^\infty(\Omega)$ is dense in $W$ and $T_{A, \psi}$ is continuous with respect to the norm of $W$. We also use the symbol $T_{A,\psi}$ for such extensions.

In the following we recall and prove the elementary properties of averaged Taylor expansions. For the approximation properties we refer to \cite{dupont_scott}.
We will need the translations $A-{\vsigma}$ of $A$ and the \emph{base} $\base{A}$ of $A$ defined as follows
\begin{align}
\label{eq:TA-def-a-trans}
&A-{\vsigma}:=\{{\valpha}-{\vsigma}\in\N^n:{\valpha}\in A\},
\\&\label{eq:TA-def-a-base}
\base{A}:=\{{\vbeta}\in\N^n\setminus A:\, \vk\in\N^n\setminus A\text{ and } \vk\le {\vbeta}\thus \vk={\vbeta}\}.
\end{align}
Fig.~\ref{fig:A-base} contains some examples of $A$ and $\base{A}$.

\begin{figure}
\begin{center}
\begin{tikzpicture}[xscale=.32, yscale=.32]
\def\NN{\draw[gray, ->] (0, 0)--(8.5, 0);\draw[gray, ->] (0, 0)--(0, 8);\foreach\x in {0, ..., 8}{\foreach\y in {0, ..., 7} {\fill[gray] (\x, \y) circle (.1cm);}}}

\begin{scope}[xshift=0cm]
\NN
\def\d{6}\def\dp{7}
\foreach\x [evaluate=\x as \yy using \d- \x] in {0, ..., \d} {\foreach\y in {0, ..., \yy} {\draw[fill=white] (\x, \y) circle (.25cm);}}
\foreach\x [evaluate=\x as \y using \dp-\x] in {0, ..., \dp}{\fill[black] (\x, \y) circle (.25cm);}
\end{scope}

\begin{scope}[xshift=11cm]
\NN
\def\dx{4}\def\dy{5}
\def\dxp{5}\def\dyp{6}
\foreach\x in {0, ..., \dx} {\foreach\y in {0, ..., \dy} {\draw[fill=white] (\x, \y) circle (.25cm);}}
\fill[black] (\dxp, 0) circle (.25cm);\fill[black] (0, \dyp) circle (.25cm);
\end{scope}

\begin{scope}[xshift=22cm]
\NN
\foreach\x in {0, ..., 3} {\foreach\y in {0, ..., 6} {\draw[fill=white] (\x, \y) circle (.25cm);}}
\foreach\x in {4, ..., 7} {\foreach\y in {0, ..., 3} {\draw[fill=white] (\x, \y) circle (.25cm);}}
\fill[black] (8, 0) circle (.25cm) (0, 7) circle (.25cm) (4, 4) circle (.25cm);
\end{scope}

\end{tikzpicture}
\end{center}
\caption[Examples of index sets $A$ (circles) and their bases $A_B$ (discs) in $2$ dimensions.]
{Examples of index sets $A$ (\tikz\draw[fill=white] (.03725, .03725) circle (.075cm);) and their bases $A_B$ (\tikz\draw[fill=black] (.03725, .03725) circle (.075cm);) in $\N^2$.}\label{fig:A-base}
\end{figure}

\subsection{General properties}

\begin{proposition}\label{thm:TA-projector}
$T_{A, \psi}$ is a projector, i.e., for all $g\in\poly_A$, $T_{A, \psi} g=g$.
\end{proposition}
\begin{proof}
The Taylor expansion at the point $\vy$ of $g\in\poly_A$ is a polynomial that has the same partial derivatives as $g$. Therefore
$$
g (\vx)=\sum_{{\valpha}\in A}\frac{(\vx-\vy)^{\valpha}}{{\valpha}!}\partial^{\valpha} g(\vy)
$$
 and since $\int\psi =1$ we have $T_{A, \psi} g (\vx)=g(\vx)$. 
\qed\end{proof}

\begin{proposition}\label{thm:TA-commutes}
For all $\vsigma\ge \vec0$ and $f\in\smooth^{\infty}$ we have
\begin{equation}\label{eq:TA-commutes}
\partial^{\vsigma} T_{A, \psi} f =T_{A-{\vsigma}, \psi}\partial^{\vsigma} f.\end{equation}
\end{proposition}
\begin{proof}
The derivatives are with respect to $\vx$ and can be computed inside the integral. We obtain
$$
\begin{aligned}
\partial^{\vsigma} T_{A, \psi} f (\vx)
&=\sum_{\substack{{\valpha}\in A\\{\valpha}\ge\vsigma}}\int_{\Omega}\psi(\vy)\frac{(\vx-\vy)^{{\valpha}-{\vsigma}}}{({\valpha}-{\vsigma})!}\partial^{\valpha} f(\vy)\, d\vy\\
&=\sum_{{\vbeta}\in A-{\vsigma}}\int_{\Omega}\psi(\vy)\frac{(\vx-\vy)^{{\vbeta}}}{{\vbeta}!}\partial^{\vbeta}\partial^{\vsigma} f(\vy)\, d\vy = T_{A-{\vsigma}, \psi}\partial^{\vsigma} f(\vx).
\end{aligned}
$$
\qed\end{proof}

\begin{lemma}
\label{thm:TA-no-derivatives}
If $\psi\in\smooth^{\maxv A}_0(\omega)$ then
\begin{equation}
\label{eq:TA-no-derivatives}
T_{A, \psi} f (\vx) =\sum_{{\valpha}\in A} C_{{\valpha}, A}\int_{\omega}\partial^{\valpha}\psi(\vy)\frac{(\vx-\vy)^{{\valpha}}}{{\valpha}!} f(\vy)\, d\vy.
\end{equation}
where $C_{{\valpha}, A}:= (-1)^{\abs{\valpha}}\sum_{\substack{{\vbeta}\in A\\{\vbeta}\ge{\valpha}}}\binom{{\vbeta}}{{\valpha}} $.

\end{lemma}

\begin{proof}
Writing \eqref{eq:TA-general} with $\vbeta$ in place of $\valpha$, integrating each term by parts, noting that the boundary terms vanish and expanding $\partial^{\vbeta} (\psi(\vy) (\vx-\vy)^{\vbeta})$ leads to
$$
\begin{aligned}
T_{A, \psi}f (\vx) &=\sum_{{\vbeta}\in A} (-1)^{\abs{\vbeta}}\sum_{\substack{{\valpha}\in A\\{\valpha}\le {\vbeta}}}\binom{{\vbeta}}{{\valpha}}\int_{\omega}\partial^{\valpha}\psi(\vy)\frac{(-1)^{\abs{\vbeta}-\abs{\valpha}}(\vx-\vy)^{{\valpha}}}{{\valpha}!} f(\vy)\, d\vy\\
	&=\sum_{{\valpha}\in A}\Big [(-1)^{\abs{\valpha}}\sum_{\substack{{\vbeta}\in A\\{\vbeta}\ge{\valpha}}}\binom{{\vbeta}}{{\valpha}}\Big ]\int_{\omega}\partial^{\valpha}\psi(\vy)\frac{(\vx-\vy)^{{\valpha}}}{{\valpha}!} f(\vy)\, d\vy.
\end{aligned}
$$
\qed\end{proof}

\begin{lemma}
\label{thm:TA-cont}
For all weights $\psi$ with $\support\psi\subseteq\omega$, and box $\eta\subset\R^n$ the operator $T_{A, \psi}$ is continuous, meaning that for all $ v$ such that $\partial^{\valpha} v\in L^q(\omega), {\valpha}\in A$ we have
\begin{equation}
\label{eq:TA-cont-full}
\norm{T_{A, \psi} v}_{p, \eta}\le \norm{\psi}_{q'}
		\sum_{{\valpha}\in A}\frac{\he^{{\valpha}+\vec1/p}}{{\valpha}!}\norm{\partial^{\valpha} v}_{q, \omega}.
\end{equation}
Moreover, if $\psi\in\smooth^{\maxv A}_0(\omega)$ then for any $v\in L^q(\omega)$
\begin{equation}
\label{eq:TA-cont-reduced}
\norm{T_{A, \psi} v}_{p, \eta}\le \Big(\sum_{{\valpha}\in A}\abs{C_{{\valpha}, A}}\frac{\he^{{\valpha}+\vec1/p}}{{\valpha}!}\norm{\partial^{\valpha}\psi}_{q', \omega}\Big)\norm{v}_{q, \omega},
\end{equation}
where $C_{{\valpha}, A}$ is from Lemma~\ref{thm:TA-no-derivatives}.
\end{lemma}

\begin{proof}
First note that for $\vy\in \eta$
$$
\Bnorm{\frac{(\cdot -\vy)^{\valpha}}{\valpha!}}_{p, \eta} \le \frac{\he^{\valpha+\vec1/p}}{\valpha!}.
$$
Let $f_{\valpha}(\vx,\vy):=\abs{\psi(\vy) (\vx-\vy)^{\valpha} (\valpha!)^{-1}\partial^{\valpha} v(\vy)}.$ 
Using Minkowski's integral inequality, see \cite[Theorem 4, p. 21]{MR0158038}, the above and H\"older's inequality we have
$$
\begin{aligned}
\norm{T_{A, \psi}v}_{p, \eta} &\le \Bnorm{\sum_{\valpha\in A}\int_{\omega} f_{\valpha}(\cdot,\vy)\,d\vy}_{p,\eta}
\le \sum_{\valpha\in A}
\Bnorm{\int_{\omega}f_{\valpha}(\cdot,\vy)\,d\vy}_{p,\eta}
\\&\le \sum_{\valpha\in A}\int_{\omega}\norm{f_{\valpha}(\cdot,\vy)}_{p,\eta}\,d\vy
\le \sum_{\valpha\in A} \frac{\he^{{\valpha}+\vec1/p}}{{\valpha}!} \norm{\psi}_{q', \omega}  \norm{\partial^{\valpha} v}_{q,\omega}.
\end{aligned}
$$
Similarly, with $f_{\valpha}(\vx,\vy)=\cabs{ C_{\valpha,A} \partial^{\valpha}\psi(\vy) (\vx-\vy)^{\valpha} (\valpha!)^{-1} v(\vy)}$ we obtain \eqref{eq:TA-cont-reduced}. 
\qed\end{proof}

\subsection{Approximation and \ref{H:approx}}

In this subsection we prove that the operators $T_{A,\psi_\omega}$ satisfy \ref{H:approx}. Here for any box $\eta$, the function $\psi_\eta$ is
\begin{equation}\label{eq:TA-psi-form}
\psi_\eta:=\measure{\eta}^{-1} (\hat\psi\circ M_\eta),
\end{equation}
where $M_\eta$ is the orientation preserving affine bijection $\eta\to[\vec0,\vec1]$ and $\hat\psi:[\vec0,\vec1]\to\R$ is a fixed function such that $\int\hat\psi=1$.
The result requires the following assumption on the $\esupp_\phi$
\def\Cconcave{C_{\Gconcave}}
\begin{enumerate}
\hitem{concave}{{\Gconcave}} each $\esupp_\phi$, $\phi\in\Phi$, is star shaped with respect to a box $\zeta_\phi$ and $\ha\le\Cconcave\hz.$
\end{enumerate}
If the $\esupp_\phi$ are boxes then taking $\zeta_\phi=\esupp_\phi$ implies \ref{H:concave} with $\Cconcave=1$.

The main idea is to specialize the approximation results from Dupont and Scott \cite{dupont_scott} to truncated boxes satisfying \ref{H:concave}.
The error bounds in \cite{dupont_scott} for a domain $\eta$ star shaped with respect to $\support\psi$ have the general form
$$
\norm{f-T_{A, \psi}f}_{p, \eta}\le C\sum_{{\valpha}\in\base{A}}\norm{\partial^{\valpha} f}_{q, \eta}.
$$
The assumptions in \cite{dupont_scott} are equivalent to $(1/q,1/p)\in R_{\abs\valpha}$, $\valpha\in\base{A}$, where
\begin{equation}\label{eq:TA-region}
R_{r}:=
\begin{cases}
	\{(\frac1q,\frac1p):\frac1p-\frac1q+\frac rn\ge0 \}\setminus\{(\frac rn,0),(1,1-\frac rn)\} & r<n\\
	[\vec0,\vec1]& r\ge n.
\end{cases}
\end{equation}
See Fig.~\ref{fig:TA-embeddings}.
Note that $R_{r}$ almost coincide with the domain of validity of the Sobolev embedding $W^{r, q}(\interior\Omega) \to L^p(\Omega)$. In fact the embedding also holds at the point $(1,1-r/n)$ \cite[Theorem 4.12 and Remark 4.13 point 3]{MR2424078}.

\begin{figure}
\mbox{}\hfill
\begin{tikzpicture}[scale=3.5]
\fill[white] (0, 0)--(.7, 0)--(1, .3)--(1, 1)--(0, 1)--cycle;
\draw[thick,fill=\FF] (.675, 0) arc [start angle=180, end angle=45, radius=.025cm]--(.9825, .2825) arc [start angle=225, end angle=90, radius=.025cm]--(1, 1)--(0, 1)--(0, 0)--cycle;
\draw (0.35, 0) node[below] {$\frac1q$} ;
\draw (0, .5) node[left] {$\frac1p$} ;
\draw[fill=black]
	(0, 0) node[below] {$(0, 0)$}
	(1, 1) node[above] {$(1, 1)$}
	(0, 1) node[above] {$(0, 1)$}
 (.85, .23) node[below right] {$\frac1p=\frac1q-\frac{\abs{\valpha}}{n}$}
;
\draw[white, fill=white] (.7, 0) circle (.02cm) node[below right, black] {$(\frac{\abs{\valpha}}{n}, 0)$};
\draw[white, fill=white] (1, .3) circle (.02cm) node[right, black] {$(1,1- \frac{ \abs{\valpha}}{n})$};
\end{tikzpicture}\hfill
\begin{tikzpicture}[scale=3.5]
\fill[white] (0, 0)--(.7, 0)--(1, .3)--(1, 1)--(0, 1)--cycle;

\draw (0.35, 0) node[below] {$\frac1q$} ;
\draw (0, .5) node[left] {$\frac1p$} ;
\draw[thick,fill=\FF] (0,0) rectangle (1,1);

\draw[fill=black]
	(0, 0) node[below] {$(0, 0)$}
	(1, 1) node[above] {$(1, 1)$}
	(0, 1) node[above] {$(0, 1)$}
 (1,0) node[below right] {$(1,0)$}
;
\end{tikzpicture}
\caption[]{
On the left, the region $R_{\abs\valpha}$ for $\abs{\valpha}<n$. On the right the case $\abs{\valpha}\ge n$.}\label{fig:TA-embeddings}
\end{figure}
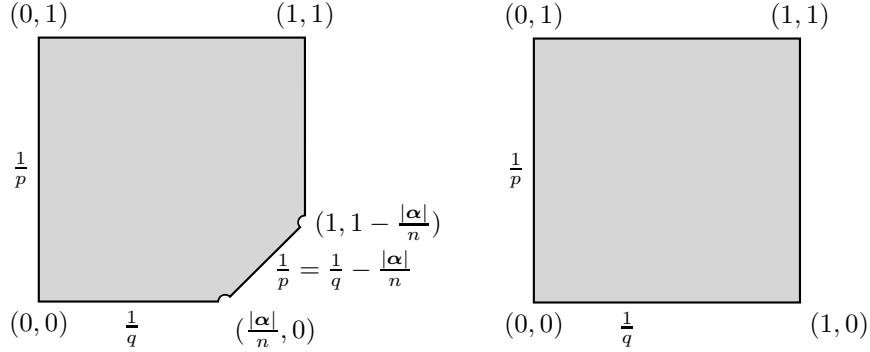

\begin{proposition}\label{thm:TA-abstract}
Let $\hat\psi\in\smooth^{\maxv A}_0([\vec0,\vec1])$ and $\int\hat\psi=1$. 
Suppose $A\subseteq\N^n$ satisfies \eqref{eq:TA-a-translation}, and that for all  $\valpha\in \base{(A-\vsigma)}$ we have $(1/q,1/p)\in R_{\abs\valpha}$.
Then there exists constants $C_{\valpha}$, such that for all truncated boxes $\eta$ star shaped with respect to a box $\zeta$ as in \ref{H:concave},  and for all $f$ with $\partial^{\valpha+\vsigma} f\in L^{q}(\eta)$ we have
\begin{equation}\label{eq:TA-abstract}
\bnorm{\partial^{\vsigma}(f-T_{A, \psi_\zeta}f)}_{p, \eta}\le \sum_{{\valpha}\in\base{(A-\vsigma)}}C_{\valpha}\bnorm{\he^{\valpha+\vpq}\partial^{\valpha+\vsigma} f}_{q,\eta}.
\end{equation}
\end{proposition}

\begin{proof}
First we consider $\vsigma=\vec0$ and a truncated box $\eta=\closure{[\vec0,\vec1]\setminus\beta}$ where $\beta$ is a box.
The diameter of $\eta$ is $\sqrt n$ because it contains two opposite vertices of $[\vec0,\vec1]$.

For a given $\zeta$ this is a special case of \cite[Theorem 4.2]{dupont_scott} and corresponds to the following substitutions and equivalency

\medskip\noindent\begin{tabular}{r|rrrrrrrrr}
Dupont-Scott & $m$ & ${\vbeta}$ & $q$ & $p_{\valpha}$ & $A^0$ & $A_-$ & $D$ & $d$ &$\mu_{\valpha}>0$\\
here & $0$ & $\vec0$ & $p$ & $q$ & $A$ & $\base{A}$ & $\eta$ & $\sqrt{n}$ & $(\frac1q,\frac1p)\in R_{\abs\valpha}$
\end{tabular}
\\Note that in \cite{dupont_scott} the floor of $r\in\R$ is written as $\ceil{r}$.

The $C_{\valpha}$ provided by \cite{dupont_scott} depend on $\psi_\zeta$, $p$, $q$.
To find $C_{\valpha}$ independent of $\zeta$ we use a compactness argument.
By Proposition~\ref{thm:TA-cont}, and since $\psi_\zeta\in\smooth^{\maxv A}_0(\zeta)$, $T_{A,\psi_\zeta}$ depends continuously on the size and position of $\zeta$.
These are described respectively by $\vh_{\zeta}\in [\Cconcave \vec1,\vec1]\subseteq\R^{n}$ and a vector in $[\vec0,\vec1-\vh_{\zeta}]\subseteq\R^{n}$.
Consequently the constants are bounded on a compact set and they have a maximum.

For other truncated boxes $\eta$ we apply a scaling and a translation. 
For $\vsigma\ne\vec0$ we use  Proposition~\ref{thm:TA-commutes} in the form $
\partial^{\vsigma}(f-T_{A,\psi}f)= g-T_{A-\vsigma,\psi}g$, $g=\partial^{\vsigma}f$.
\qed\end{proof}

We next consider a constant weight.  Proposition~\ref{thm:TA-abstract} does not apply since $\psi\not\in\smooth^{\maxv A}_0$, nevertheless we obtain an error bound with an explicit constant.

\begin{proposition}\label{thm:TA-explicit}
Let $\eta$ be a box. For $1\le p, q\le\infty$, $d\ge n$, $A =\{{\valpha}\in\N^n:\abs{\valpha}\le d\}$ and
$\psi_\eta (\vx): =\measure{\eta}^{-1}\ind_{\eta} (\vx)$
we have
\begin{equation}\label{eq:TA-explicit}
\bnorm{\partial^{\vsigma}(f-T_{A, \psi_\eta}f)}_{p,\eta}\le\sum_{\abs{\valpha}=d+1-\abs{\vsigma}}\frac {(d+1-\abs{\vsigma})n^{\frac {d-1}2}}{\valpha!}\bnorm{\he^{\valpha+\vpq}\partial^{\valpha+\vsigma} f}_{q,\eta}.
\end{equation}
\end{proposition}

\begin{proof}
Suppose $\vsigma=\vec0$.
It is enough to prove the case $p=\infty$ and $q=1$
because
$\norm{f-T_{A,\psi_\eta}f}_{p,[\vec0,\vec1]}\le \norm{f-T_{A,\psi_\eta}f}_{\infty,[\vec0,\vec1]}$ and $\norm{\partial^{\valpha}f}_{1,[\vec0,\vec1]}\le \norm{\partial^{\valpha}f}_{q,[\vec0,\vec1]}$.
Using Sobolev representation \cite[Section~3]{dupont_scott}, that holds also for $\psi\not\in\smooth^{\vmax A}_0$, we have
\begin{equation}\label{eq:TA-explicit-1}
(f-T_{A, \psi_\eta}f)(\vx)=\sum_{\abs{\valpha}=d+1}\frac{d+1}{{\valpha}!}\int_{[\vec0,\vec1]}  K_{\valpha}(\vx,\vy) \partial^{\valpha} f(\vy)\,d\vy
\end{equation}
where $K_{\valpha}(\vx,\vy):=(\vx-\vy)^{\valpha} \int_{[0, 1]} s^{-n-1}\psi_\eta(\vx+s^{-1}(\vy-\vx))\, ds$.
The integrand in $K_{\valpha}$ is $0$ for $\vx+s^{-1}(\vy-\vx)\not\in\support\psi_\eta$, in particular for $s\le n^{-1/2}\norm{\vx-\vy}$.
For $\vx,\vy\in[\vec0,\vec1]$ we have $\norm{\vx-\vy}\le n^{1/2}$ and we get
$$
\abs{K_{\valpha}(\vx,\vy)}\le \norm{\vx- \vy}^{d+1} \int_{n^{-1/2}\norm{\vx-\vy}}^1\hspace{-3em} s^{-n-1}\,ds\le n^{n/2-1}\norm{\vx-\vy}^{d+1-n}\le n^{(d-1)/2}.
$$
Inserting this estimate in \eqref{eq:TA-explicit-1} gives the result for $p=\infty$ and $q=1$.
The case $\vsigma\ne\vec0$ follows as in Proposition~\ref{thm:TA-abstract}.
\qed\end{proof}

\begin{theorem}\label{thm:TA-approx}
Let $\hat\psi\in\smooth^{\maxv A}_0([\vec0,\vec1])$ and $\int\hat\psi=1$. 
Suppose $A\subseteq\N^n$ satisfies \eqref{eq:TA-a-translation}, and that for all  $\valpha\in \base{(A-\vsigma)}$ we have $(1/q,1/p)\in R_{\abs\valpha}$. Assume \ref{H:concave}.
Then there exists a constant $\Capprox$ such that the collection of operators $T_{A, \psi_\omega}$ satisfies \ref{H:approx} with the following pairs $(K_{\vbeta}, {\vgamma})$
\begin{align}
\label{eq:TA-approx-full}&K_{\vbeta} =\bigcup_{{\vb}\in (A-\vbeta)} (\base{(A-\vb-\vbeta)}+{\vb})&&&\text{and}&&&{\vgamma}=-\vec1/q, \\
\label{eq:TA-approx-reduced}&K_{\vbeta} =\base{(A-{\vbeta})}&&&\text{and}&&&{\vgamma}=-\maxv A-\vec1/q.
\end{align}
In particular for $A=\{{\valpha}\in\N^n:\abs{\valpha}\le d\}$, \eqref{eq:TA-approx-full} becomes 
$$K_{\vbeta} =\{\vk\ge\vbeta:\abs{\vk}=d+1\}, \qquad\vgamma=- \vec1/q,$$
and if in addition the $\esupp_\phi$ are boxes, i.e. $\Cconcave=1$, and $d-\abs{\vsigma}\ge n$ then 
\begin{equation}\label{eq:TA-approx-explicit}
\Capprox =2\frac{(2n^{\frac 32})^{d+1}}{(d-\abs{\vsigma})!}.
\end{equation}
\end{theorem}

\begin{proof}
We first consider the case $\vbeta=\vec0$. For each $\omega\in\mesh$, we need to consider the approximation for both $\eta=\omega$ and $\eta=\esupp_\phi$, $\phi\in\actives_\omega$.
If $\eta=\omega$ then \ref{H:approx} follows from equation~\eqref{eq:TA-abstract} in Proposition~\ref{thm:TA-abstract}.

Suppose $\eta=\esupp_\phi$. By Proposition~\ref{thm:TA-projector}, $T_{A, \psi_\omega}$ is a projector on $\poly_A$ and for all $g$ we have
\begin{equation}\label{eq:TA-approx-0}
f-T_{A, \psi_\omega} f=(f-g) + T_{A, \psi_\omega} (f-g).
\end{equation}
By \eqref{eq:TA-cont-full} and $\norm{\psi_\omega}_{q'}=\ho^{-\vec1/q}\cnorm{\hat\psi}_{q'}$ we deduce
\begin{equation}\label{eq:TA-approx-1}
\norm{f-T_{A, \psi_\omega} f}_{p,\eta}\le\norm{f-g}_{p,\eta}+ \cnorm{\hat\psi}_{q'}\frac{\he^{\vec1/q}}{\ho^{\vec1/q}}\sum_{{\vb}\in A}\frac{\he^{{\vb}+\vpq}}{{\vb}!}\bnorm{\partial^{\vb} (f-g)}_{q, \eta}.
\end{equation}
We fix $g:=T_{A, \psi_\zeta} f\in\poly_{A}$ where $\zeta\subseteq\eta$ is a box satisfying \ref{H:concave}, use Proposition~\ref{thm:TA-abstract}, and obtain 
$$
\begin{aligned}
\norm{f-T_{A, \psi_\omega} f}_{p,\eta}\le& \sum_{\valpha\in \base{A}} C_{\valpha,p}\he^{\valpha+\vpq}\bnorm{\partial^{\valpha}f}_{q,\eta}
\\&+\cnorm{\hat\psi}_{q'}\frac{\ho^{-\vec1/q}}{\he^{-\vec1/q}}\sum_{\vb\in A}\sum_{\valpha\in\base{(A-\vb)}}C_{\valpha,q}\frac{\he^{\vb+\valpha+\vpq}}{\vb!}\bnorm{\partial^{\valpha+\vb}f}_{q,\eta}
\\\le& \Capprox \frac{\ho^{-\vec1/q}}{\he^{-\vec1/q}} \sum_{\vk\in K_{\vec0}} \he^{\vk+\vpq} \bnorm{\partial^{\vk}f}_{q,\eta},
\end{aligned}
$$
where we used $\ho^{-\vec1/q}\he^{\vec1/q}\ge1$, $\vbeta=\vec0$, $K_{\vec0}$ as in \eqref{eq:TA-approx-full} and $\Capprox$ equals the maximum of the $C_{\valpha,p}$, $\valpha\in \base{A}$ and $\cnorm{\hat\psi}_{q'}C_{\valpha,q}(\vb!)^{-1}$, $\vb\in A$, $\valpha\in \base{(A-\vb)}$ times the maximum number of repetitions of a derivative of $f$.

To obtain \eqref{eq:TA-approx-explicit} we use Proposition~\ref{thm:TA-explicit} instead of Proposition~\ref{thm:TA-abstract}.
We notice that $\cnorm{\hat\psi}_{q'}=1$ and re-index the doubles sum on $\vc=\valpha+\vb$ as in
\begin{equation}\label{eq:TA-approx-3}
\begin{aligned}
&\sum_{\abs{\vb}\le d}\quad\sum_{\abs{\valpha}=d+1-\abs{\vb}} \frac {(d+1-\abs{\vb})n^{\frac {d-1}2}}{\valpha!\ \vb!} \he^{\vb+\valpha+\vpq} \bnorm{\partial^{\valpha+\vb}f}_{q,\eta}
\\&\qquad=\sum_{\abs{\vc}=d+1}\sum_{\vb\le\vc}\frac {(d+1-\abs{\vb})n^{\frac {d-1}2}}{(\vc-\vb)!\ \vb!} \he^{\vc+\vpq} \norm{\partial^{\vc}f}_{q,\eta}
\\&\qquad\le \sum_{\abs{\vc}=d+1} \frac {(d+1)2^{d+1} n^{\frac {d-1}2}}{\vc!} \he^{\vc+\vpq} \norm{\partial^{\vc}f}_{q,\eta},
\end{aligned}
\end{equation}
where we used 
$$
\sum_{\vb\le \vc} \frac{\vc!}{(\vc-\vb)!\vb!}=\sum_{\vb\le \vc}\prod_{i=1}^n\binom{c_i}{b_i}=\prod_{i=1}^n\sum_{b_i\le c_i}\binom{c_i}{b_i}= 2^{\abs{\vc}}=2^{d+1}.
$$
The bounds in Proposition~\ref{thm:TA-explicit} for $\norm{f-g}_{p,\eta}$ and $\norm{f-g}_{q,\eta}$ are equal. 
Since the latter appears in the sum in \eqref{eq:TA-approx-1} for $\vb=0$ then we can bound $\norm{f-T_{A,\psi_\omega}f}_{p,\eta}$ with 2 times the bound in \eqref{eq:TA-approx-3}. Using also that $(\vc !)^{-1}\le n^{d+1} (d+1)!^{-1}$ we obtain \eqref{eq:TA-approx-explicit} for $\vsigma=\vec0$.

To obtain \eqref{eq:TA-approx-reduced} we start from \eqref{eq:TA-cont-reduced}, 
set $g=T_{A, \psi_\zeta} f\in\poly_{A}$ where $\zeta\subseteq\eta$ is a box satisfying \ref{H:concave}, use $\cnorm{\partial^{\valpha}\psi_\omega}_{q'}=\ho^{-\vec1/q}\ho^{-\valpha}\cnorm{\partial^{\valpha}\hat\psi}_{q'}$, Proposition~\ref{thm:TA-abstract} and a scaling argument and obtain
\begin{equation}\label{eq:TA-approx-2}
	\begin{aligned}
\bnorm{f-T_{A, \psi_\omega} f}_{p,\eta}&\le
\Big(\sum_{\valpha\in A} \frac{\abs{C_{\valpha,A}}}{\valpha!}\frac{\ho^{-\valpha-\vec1/q}}{\he^{-\valpha-\vec1/q}}\he^{\vpq}\cnorm{\partial^{\valpha}\hat\psi}_{q'} \Big) \bnorm{f-g}_{q,\eta}
\\&\le \Capprox \frac{\ho^{-\maxv A-\vec1/q}}{\he^{-\maxv A-\vec1/q}} \sum_{\valpha\in \base{A}} \he^{\valpha+\vpq} \bnorm{\partial^{\valpha}f}_{q,\eta}
\end{aligned}\end{equation}
where $\Capprox=(\sum_{\valpha\in A}  \abs{C_{\valpha,A}} \valpha!^{-1}\cnorm{\partial^{\valpha}\hat\psi}_{q'})\max_{\valpha\in \base{A}}C_{\valpha}$.

The case $\vbeta=\vsigma\ne\vec0$ follows from Proposition~\ref{thm:TA-commutes}.
\qed\end{proof}

\section{Sobolev and reduced seminorms}\label{sec:estimates}

We define two families of operators $\paop^S_\omega$ and $\paop^R_\omega$ of the form $T_{A,\psi_\omega}$ from which error bounds are derived in term of Sobolev and reduced seminorms, respectively.
For each case we summarize the assumptions required by Theorems~\ref{thm:global-ani}, \ref{thm:global-iso} and \ref{thm:global-inf}.

The family $\paop^S_\omega$ is defined by $A=\{{\valpha}\in\N^n:\abs{\valpha}\le d\}$, $\psi_\omega=\measure{\omega}^{-1}\ind_{\omega}$ and satisfy \ref{H:approx} according to \eqref{eq:TA-approx-full} with
\begin{equation}\label{eq:K-sobolev}
\vgamma=-\vec1/q,\ \  K_{\vec0}=\{\vk :\, \abs{\vk}= d+1\},\ \  K_{\vsigma}=\{\vk :\, {\vsigma}\le \vk\text{ and }\abs{\vk}= d+1\}.
\end{equation}
The family $\paop_\omega^R$ is defined by $A=\{\valpha \in\N^n: \valpha\le\vd\}$, $\psi_\omega$ as in \eqref{eq:TA-general} for a fixed  $\hat\psi\in\smooth^{\vd}_0([\vec0,\vec1])$ and satisfy \ref{H:approx} according to \eqref{eq:TA-approx-reduced} with
\begin{equation}\label{eq:K-reduced}
\begin{aligned}
	&\vgamma=-\vd-\vec1/q,\\
&K_{\vec0}=\{(0, \dots, d_i+1, \dots, 0):\, i=1, \dots, n\},
\\&K_{\vsigma}=\{(\sigma_1, \dots, d_i+1, \dots, \sigma_n):\, i=1, \dots, n\}.
\end{aligned}\end{equation}
A graphical representation of the index sets for $\paop^S_\omega$ and $\paop^R_\omega$ is shown in Fig.~\ref{fig:a-k-pairs-sobolev} and \ref{fig:a-k-pairs-reduced}.
Remembering Propositions~\ref{thm:mesh-assumption}, \ref{H:mesh} is implied by different mesh properties depending on whether $p\le q$ or $p>q$.
This leads to the following theorem.

\begin{theorem}\label{thm:sobolev}
Suppose that \ref{H:polynomial}, \ref{H:lambda}, \ref{H:phi}, \ref{H:esupp}, \ref{H:duals}, \ref{H:concave} hold and $(1/q,1/p)\in R_{d+1-\abs{\vsigma}}$ for $\paop^S_\omega$ and in $R_{\min (\vd-{\vsigma})+1}$ for $\paop^R_\omega$.
Then \eqref{eq:global-ani}, \eqref{eq:global-iso}, \eqref{eq:global-inf-ani} and \eqref{eq:global-inf-iso}  hold under the additional assumptions listed in the following table.
\begin{center}
\newdimen\LS
\LS=1.1ex
\def\level#1{\hspace{#1\LS}}
\setlength\tabcolsep{0.77ex}
\noindent\begin{tabular}{|l|ll|l|l|ll|l|}
\multicolumn{2}{c}{} &\multicolumn{3}{c}{$\paop^S_\omega$} &&\multicolumn{2}{c}{$\paop^R_\omega$}\\
\cline{3-5}\cline{7-8}
\multicolumn{2}{c}{} &\multicolumn{3}{|c|}{$K_{\vec0}$, $K_{\vsigma}$ from \eqref{eq:K-sobolev}} &&\multicolumn{2}{|c|}{$K_{\vec0}$, $K_{\vsigma}$ from \eqref{eq:K-reduced}}\\
\multicolumn{2}{c|}{} &$p\le q<\!\infty$& $p\le q=\!\infty$ & $q<p\phantom{<\infty}$ &&\multicolumn{1}{|l|}{$p\le q\phantom{<\infty} $} & $q<p\phantom{<\infty}$\\
\cline{3-5}\cline{7-8}
\multicolumn{2}{c}{error bound} &\multicolumn{3}{c}{assumptions} &&\multicolumn{2}{c}{assumptions}\\
\cline{1-1}\cline{3-5}\cline{7-8}
\eqref{eq:global-ani} && \multicolumn{1}{|l|}{\ref{H:mesh-num} }& &\notableentry&&\multicolumn{1}{|l|}{\ref{H:mesh-length}} &\notableentry\\
\eqref{eq:global-inf-ani} &&\multicolumn{1}{|l|}{\!impossible\!}& &\notableentry&&\multicolumn{1}{|l|}{\ref{H:mesh-length}} &\notableentry\\
\eqref{eq:global-iso}&&\multicolumn{1}{|l|}{ \ref{H:mesh-num}, \ref{H:isotropicA}}& \ref{H:isotropicA} & \ref{H:mesh-length}, \ref{H:isotropicA}
	&&\multicolumn{1}{|l|}{\ref{H:mesh-length}, \ref{H:isotropicA}} &  \ref{H:mesh-length}, \ref{H:isotropicA}\\
\eqref{eq:global-iso}, $\vsigma=\vec0$ && \multicolumn{1}{|l|}{\ref{H:mesh-num}}&& \ref{H:mesh-length}&&\multicolumn{1}{|l|}{\ref{H:mesh-length}}&\ref{H:mesh-length}\\
\eqref{eq:global-inf-iso} && \multicolumn{1}{|l|}{\ref{H:mesh-num}, \ref{H:isotropicA}}& \ref{H:isotropicA} & \ref{H:mesh-length}, \ref{H:isotropicA}
	&&\multicolumn{1}{|l|}{\ref{H:mesh-length}, \ref{H:isotropicA}} &  \ref{H:mesh-length}, \ref{H:isotropicA}\\
\eqref{eq:global-inf-iso}, $\vsigma=\vec0$ && \multicolumn{1}{|l|}{\ref{H:mesh-num}}&& \ref{H:mesh-length}&&\multicolumn{1}{|l|}{\ref{H:mesh-length}}&\ref{H:mesh-length}\\
\cline{1-1}\cline{3-5}\cline{7-8}
\multicolumn{2}{c}{constant} &\multicolumn{3}{c}{substitutions} &&\multicolumn{2}{c}{substitutions}\\
\cline{1-1}\cline{3-5}\cline{7-8}
$\Cmesh$ && \multicolumn{1}{|l|}{$\Cnum^{\frac1q}$} &$1$ &$\Clength^{n \pq_-+\frac np}$&&\multicolumn{1}{|l|}{$\Clength^{\abs{\vd}+n}$} & $\Clength^{\abs{\vd}+2n}$\\
 $\CisotropicO$ && \multicolumn{1}{|l|}{$1$ }& $1$ & $\Clength\CisotropicA$&& \multicolumn{1}{|l|}{$1$ }& $\Clength\CisotropicA$\\
$\Capprox$ &&\multicolumn{3}{|l|}{if $d\!-\!\abs\vsigma\!\ge\!n$ and $\Cconcave\!=\!1$ see \eqref{eq:TA-approx-explicit}} &&\multicolumn{1}{|l|}{}&\\
\cline{1-1}\cline{3-5}\cline{7-8}
\end{tabular}\end{center}
Empty cells mean the estimate holds without additional assumptions. A stroked cell means that the estimate does not apply and one should refer to Theorem~\ref{thm:global}.The table also lists possible substitutions for $\Cmesh$, $\CisotropicO$ and $\Capprox$.
\end{theorem}

\begin{example}
For $\vsigma=\vec0$ and $p = q = 2$ one obtains \eqref{eq:global-iso} for $\paop^S_\omega$ from \ref{H:polynomial}, \ref{H:lambda}, \ref{H:phi}, \ref{H:esupp}, \ref{H:duals}, \ref{H:concave} plus \ref{H:mesh-num} with the substitution from the same column in $\Ctot$, i.e., $\Cmesh=\Cnum^{1/2}$, $\CisotropicO=1$ and $\cardinality{K_{\vec0}}=\binom{d+n-1}{n-1}$ giving
\begin{align*}
\Ctot = 4\binom{d+n-1}{n-1}^{\frac12}\Clambda\Cphi\Capprox\Cesupp^{\frac n2}\Cduals\Cnum^{\frac12}.
\end{align*}
\end{example}

\begin{remark}
For $p>q$ element shape regularity, i.e.\ \ref{H:isotropicO}, is required indirectly as it is implied by the assumptions \ref{H:mesh-length} and \ref{H:isotropicA}. For $p\le q$ element shape regularity is not required and shape regularity of the $\esupp_\phi$'s is sufficient. Finally for $q=\infty$ and $\vsigma=\vec0$ no shape regularity assumption is required.
\end{remark}

\begin{figure} 
\begin{center}
\begin{tikzpicture}[xscale=.32, yscale=.32]
\def\NN{\draw[gray, ->] (0, 0)--(10, 0);\draw[gray, ->] (0, 0)--(0, 8);\foreach\x in {0, ..., 9}{\foreach\y in {0, ..., 7} {\fill[gray] (\x, \y) circle (.1cm);}}}
\begin{scope}[xshift=0cm]
\node[anchor=north west] at (0,-.5) {$\begin{aligned}&A=\{\vk:\abs{\vk}\le6\}\\&K_{\vec0}=\{\vk:\abs{\vk}=7\}\end{aligned}$};
\NN
\def\d{6}\def\dp{7}
\foreach\x [evaluate=\x as \yy using \d-\x] in {0, ..., \d} {\foreach\y in {0, ..., \yy} {\draw[fill=white] (\x, \y) circle (.25cm);}}
\foreach\x [evaluate=\x as \y using \dp-\x] in {0, ..., \dp}{\fill[black] (\x, \y) circle (.25cm);}
\draw[fill=red] (0, 0) circle (.25cm);
\end{scope}
\begin{scope}[xshift=18cm, yshift=0cm]
\node[anchor=north west] at (0,-.5) {$\begin{aligned}&A=\{k:\abs{\vk}\le 6\}\\ &K_{(2,3)}=\{(2,5),\,(3,4),\,(4,3)\}\end{aligned}$ };
\NN
\def\d{6}\def\dp{7}
\def\sx{2}\def\sy{3}
\pgfmathsetmacro\xl{\dp-\sy}
\foreach\x [evaluate=\x as \yy using \d-\x] in {0, ..., \d} {\foreach\y in {0, ..., \yy} {\draw[fill=white] (\x, \y) circle (.25cm);}}
\foreach\x [evaluate=\x as \y using \dp-\x] in {\sx, ..., \xl}{\fill[black] (\x, \y) circle (.25cm);}
\draw[fill=red] (\sx, \sy) circle (.25cm);

\end{scope}
	\pgfresetboundingbox 
	\path[use as bounding box] (-.25,-2.5) rectangle (28,8);
\end{tikzpicture}
\end{center}
\caption[Considered pairs of index sets $A$ (circles) and $K$ discs in $2$ dimensions.]
{Some examples of $A$(\tikz\draw[fill=white] (.03725, .03725) circle (.075cm);) and $K_{\vbeta}$(\tikz\draw[fill=black] (.03725, .03725) circle (.075cm);) for Sobolev seminorms and given ${\vbeta}$(\tikz\draw[fill=red] (.03725, .03725) circle (.075cm);).}\label{fig:a-k-pairs-sobolev}
\end{figure}
\begin{figure}
	\begin{center}
	\begin{tikzpicture}[xscale=.32, yscale=.32]
\def\NN{\draw[gray, ->] (0, 0)--(10, 0);\draw[gray, ->] (0, 0)--(0, 8);\foreach\x in {0, ..., 9}{\foreach\y in {0, ..., 7} {\fill[gray] (\x, \y) circle (.1cm);}}}
	\begin{scope}[xshift=0cm]
	\NN
	\def\dx{4}\def\dy{5}
	\def\dxp{5}\def\dyp{6}
	\foreach\x in {0, ..., \dx} {\foreach\y in {0, ..., \dy} {\draw[fill=white] (\x, \y) circle (.25cm);}}
	\fill[black] (\dxp, 0) circle (.25cm);\fill[black] (0, \dyp) circle (.25cm);
	\draw[fill=red] (0, 0) circle (.25cm);
\node[anchor=north west] at (0,-.5) {$\begin{aligned}&A=\{k\le (4, 5)\} \\ &K_{\vec0}=\{(0,6),\,(5,0)\}\end{aligned}$ };
	\end{scope}
	\begin{scope}[xshift=18cm, yshift=-0cm]
	\NN
	\def\dx{4}\def\dy{5}
	\def\dxp{5}\def\dyp{6}
	\foreach\x in {0, ..., \dx} {\foreach\y in {0, ..., \dy} {\draw[fill=white] (\x, \y) circle (.25cm);}}
	\fill[black] (\dxp, 4) circle (.25cm);\fill[black] (3, \dyp) circle (.25cm);
	\draw[fill=red] (3, 4) circle (.25cm);
\node[anchor=north west] at (0,-.5) {$\begin{aligned}&A=\{k\le (4, 5)\}\\ &K_{(3,4)}=\{(3,6),\,(5,4)\}\end{aligned}$ };
	\end{scope}
	\pgfresetboundingbox 
	\path[use as bounding box] (-.25,-2.5) rectangle (28,8);
	\end{tikzpicture}
\end{center}
\caption[Considered pairs of index sets $A$ (circles) and $K$ discs in $2$ dimensions.]
{Examples of $A$(\tikz\draw[fill=white] (.03725, .03725) circle (.075cm);) and $K_{\vbeta}$(\tikz\draw[fill=black] (.03725, .03725) circle (.075cm);) for reduced seminorms and given ${\vbeta}$(\tikz\draw[fill=red] (.03725, .03725) circle (.075cm);).}\label{fig:a-k-pairs-reduced}
\end{figure}

\section{Analytic functions and exponential convergence}
\label{sec:EXP}

In this section we consider a sequence of spline spaces $\{\spline_d=\SPAN\Phi_d\}_{d \ge \abs{\vsigma}+n}$ of  degree $\vd= (d, \dots, d)$ and defined on the same domain $\Omega$.
To each space we associate a corresponding operator $\aop_d$ and we study the behavior of $\norm{\partial^{\vsigma}(f-\aop_d f)}_p$ as a function of $d$.
We start by recalling some basic properties of analytic functions.

\subsection{Some properties of analytic functions}

For more  on the following material, see for example \cite{MR847923}.

\begin{definition}\label{def:analytic}
A function $f:U\subseteq\C^n\to\C$ or $f:U\subseteq\R^n\to\R$ is \emph{analytic} at $\vx$ if there is an open neighbourhood $U_x$ of $\vx$ where the sequence
$$ T_{A_d, \delta_x}f(\vy):=\sum_{{\valpha}\le d}\frac{(\vy-\vx)^{\valpha}}{{\valpha}!}\partial^{\valpha} f (\vx)$$
converges uniformly to $f$ as $d\to\infty$. A function is said to be analytic on $U$ if it is analytic at $\vx$ for all $\vx\in U$.
\end{definition}

If $f:U\subset\C\to\C$ is analytic on $ D_{x, r}:=\{z\in\C:\abs{x-z}\le r\}$ then the Cauchy formula states
$$
f(x)=\frac1{2\pi i}\int_{\boundary D_{x, r}}\frac{f( z)}{z-x}\, dz.
$$
This has a multivariate analogue \cite[Theorem1.3]{MR847923} that follows by applying the above formula to each coordinate.
Given $\vx\in\C^n$ and $\vec0<\vr\in\R^n$ let
\begin{align*}
D_{\vx, \vr}&:=D_{x_1, r_1}\times\dots\times D_{x_n, r_n},&&&
\tilde\partial D_{\vx, \vr} &:=\boundary D_{x_1, r_1}\times\dots\times\boundary D_{x_n, r_n}.
\end{align*}
For $f:U\subset\C^n\to\C$ analytic on $D_{\vx, \vr}$ we have
\begin{equation}\label{eq:EXP-cauchy}
f(\vx)=(2\pi i)^{-n}\int_{\tilde\partial D_{\vx, \vr}}\frac {f(\vz)}{(\vz-\vx)^{\vec1}}\, d\vz.
\end{equation}
The Cauchy formula implies the following proposition, see also \cite[Theorem~1.6]{MR847923}.

\begin{proposition}\label{thm:EXP-derivatives}
Let $\Omega\subset \C^n$ be any set. If $f$ is analytic on $\Delta\subseteq\C^n$ containing $\Omega$ and 
\begin{equation}\label{eq:EXP-derivatives-distance}
\vr(\vx) :=\maxv\{\vw: D_{\vx, \vw}\subseteq\Delta\}>\vec0,\qquad\forall \vx \in \Omega
\end{equation}
then for all $\vx\in\Omega$
$$
\abs{\partial^{\valpha} f(\vx)}\le\frac{{\valpha}!}{\vr(\vx)^{{\valpha}}}\norm{f}_{\infty, \Delta}.
$$
\end{proposition}

\begin{proof}
Differentiating the Cauchy formula ${\valpha}$ times with respect to $\vx$ we have
$$
\begin{aligned}
\partial^{\valpha} f(\vx)&=\partial^{\valpha} (2\pi i)^{-n}\int_{\tilde\partial D_{\vx, \vr(\vx)}}\frac {f(\vz)}{(\vz-\vx)^{\vec1}}\, d\vz = (2\pi i)^{-n} {\valpha}!\int_{\tilde\partial D_{\vx, \vr(\vx)}}\frac {f(\vz)}{(\vz-\vx)^{{\valpha}+\vec1}}\, d\vz
\end{aligned}
$$
and consequently
$$
\abs{\partial^{\valpha} f(\vx)}\le (2\pi)^{-n} {\valpha}!\norm{f}_{\infty, \Delta}\int_{\tilde\partial D_{\vx, \vr(\vx)}}\babs{(\vz-\vx)^{-\valpha-\vec1}}\, d\vz.
$$
Since the integral equals $\vr(\vx)^{-{\valpha}} (2\pi)^{n}$, the result follows.
\qed\end{proof}

Proposition~\ref{thm:EXP-derivatives} does not apply to real analytic function $U\subset\R^n\to\R$ and we use the following proposition:

\begin{proposition}\label{thm:EXP-extension}
All real analytic functions $f:\Omega\subset\R^n\to\R$ admit an analytic extension to a  neighbourhood $\Delta\subset\C^n$ containing $\Omega$ and such that $\vr(\vx)$ defined in \eqref{eq:EXP-derivatives-distance} is strictly positive on $\Omega$.
\end{proposition}

\begin{proof}
For $\vx\in\Omega$, let $r(\vx)>0$ be the radius of convergence of the Taylor series of $f$ around $\vx$.
The series is absolutely convergent on $\{\vy \in\R^n:\norm{\vx-\vy}< r(\vx)\}$ and thus on $ \interior{D}_{\vx, r(\vx)} := \{\vy \in\C^n:\norm{\vx-\vy}< r(\vx)\}$.
This implies that $f$ admits an analytic continuation to
 $$\Delta:=\bigcup_{\vx\in\Omega} {D}_{\vx, r(\vx)/2}\subset\bigcup_{\vx\in\Omega} \interior{D}_{\vx, r(\vx)}$$
and $\vmin\vr(\vx)\ge r(\vx)/\sqrt n>0$ for all $\vx\in\Omega$.
\qed\end{proof}

\subsection{Exponential convergence}

The following theorem implies that the error decreases exponentially as the degree increases, provided that the space resolution is sufficiently small.

\begin{theorem}\label{thm:EXP-convergence}
Let $\{\Phi_d\}$, $\{\aop_d\}$, $d\ge\abs{\vsigma}+n$ be a sequence of generating systems such that $\esupp_\phi$ is a box for all $\phi\in\bigcup_d\Phi_d$.
Suppose \ref{H:polynomial}, \ref{H:lambda}, \ref{H:phi}, \ref{H:esupp}, \ref{H:duals} \ref{H:isotropicA} hold with constants satisfying
\begin{align}\label{eq:EXP-growth}
\Clambda\Cphi\Cesupp^{\abs{\vsigma}}\Cactives\CisotropicA^{\abs{\vsigma}}\le\Cgrowth\Bgrowth^{d+1} (d+1)^\Egrowth,
\end{align}
for some $\Cgrowth, \Bgrowth, \Egrowth>0$.
Then for all analytic function $f:\Omega\subseteq\R^n\to\R$ and  $\Delta$ and $\vr$ as in  Proposition~\ref{thm:EXP-extension} we have
$$
\norm{\partial^{\vsigma}(f-\aop_d f)}_\infty\le C_{{\vsigma}} {(d+1)^{n+\Egrowth+\abs{\vsigma}}} \hat\tau^{d+1-\abs{\vsigma}}\norm{f}_{\infty, \Delta}
$$
where 
$$\hat\tau:=2 n^{\frac 32}\Bgrowth\ \sup_{d}\Bnorm{\frac { \vmax{{\hmd}} }{\vmin{\vr}}}_\infty,\qquad  
C_{\vsigma}:= 4 \Cgrowth  (2 n^{\frac32}\Bgrowth)^{\abs{\vsigma}} \bnorm{(\vmin\vr)^{-\abs\vsigma}}_\infty. 
$$
\end{theorem}

\begin{proof}
Since the $\esupp_\phi$ are boxes, \ref{H:concave} holds with $\Cconcave=1$.
Theorem~\ref{thm:sobolev} applies and \eqref{eq:global-inf-iso} holds with $K_{\vec0}, K_{\vsigma}$ as in \eqref{eq:K-sobolev} for $p=q=\infty$.
Bounding $\abs{\partial^{\vk} f(\vx)}$ as in Proposition~\ref{thm:EXP-derivatives} and noting $K_{\vec0}\cup K_{\vsigma}=K_{\vec0}$ we have 
$$
\norm{\partial^{\vsigma}(f-\aop_d f)}_\infty\le C\cardinality{K_{\vec0}}\max_{\substack{\phi\in\Phi_d\\\abs{\vk}=d+1}}\Big\{\Bnorm{\frac {\vk!(\vmax\ha)^{d+1-\abs{\vsigma}}}{\vr^{\vk}}}_{\infty, \esupp_\phi}\Big\}\norm{f}_{\infty, \Delta}
$$
where using the substitutions from Theorem~\ref{thm:sobolev}, $\Capprox$ from \eqref{eq:TA-approx-explicit} and \eqref{eq:EXP-growth}
$$
C=2\Clambda\Cphi\Capprox\Cesupp^{\abs\vsigma}\Cduals\Cmesh \big (\CisotropicO^{\abs{\vpq_-}}+\CisotropicA^ {\abs{\vsigma}+\abs{\vpq_-}} )\le4 \Cgrowth \frac{(d+1)^{\Egrowth} (2 n^{\frac 32}\Bgrowth )^{d+1}}{(d-\abs{\vsigma})!}.
$$
The result follows using $\cardinality{K_{\vec0}}\le(d+1)^{n-1}$, $\vk!/(d-\abs\vsigma)!\le (d+1)^{\abs\vsigma+1}$ and
$$
\frac {(\vmax\ha)^{d+1-\abs{\vsigma}}}{\vr^{\vk}}\le \frac{1}{(\vmin\vr)^{\abs\vsigma}}\frac{\vmax\hmd^{d+1-\abs\vsigma}}{(\vmin\vr)^{d+1-\abs\vsigma}}\le\frac{1}{(\vmin\vr)^{\abs\vsigma}}\frac{ \hat\tau^{d+1-\abs{\vsigma}}} {(2n^{\frac32}\Bgrowth)^{d+1-\abs{\vsigma}}}.
$$
\qed\end{proof}

If $\hat\tau<1$ then $(d+1)^{n+\Egrowth+\abs\vsigma} \hat\tau^{d+1-\abs\vsigma}$ decreases exponentially in $d$.
More precisely for all $1>\tau>\hat\tau$ we have
\begin{equation}\label{eq:EXP-clean}
\norm{\partial^{\vsigma}(f-\aop_d f)}_\infty\le C_{\tau} \tau^{d+1-\abs{\vsigma}}\norm{f}_{\infty, \Delta},
\end{equation}
where
$C_{\tau}:=C_{{\vsigma}}\max_{d\ge\abs {\vsigma}+n} \{(d+1)^{n+\Egrowth+\abs{\vsigma}} ( \hat\tau/\tau )^{d+1-\abs{\vsigma}} \}$ is bounded independently of $d$.
This shows exponential convergence.

Note that by $h$-refinement of the $\{\Phi_d\}$ it is always possible to obtain $\hat\tau<1$.
We are able to prove exponential convergence on the space dimension by linking the space resolution with the cardinality of $\Phi$ as done in the following proposition.

\begin{proposition}\label{thm:EXP-cardinality}
If $\SPAN\Phi\supseteq\poly_{\vd}$ and the $\esupp_\phi$ are boxes then
$$
 (\vd+\vec1)^{\vec1}\int_{\Omega}\hm^{-\vec1} d\vx \le \cardinality{\Phi}\le\Cgraded\Cduals\int_{\Omega}\hm^{-\vec1} d\vx
$$
where the upper bound requires \ref{H:duals} and  that for all $\omega\in\mesh$
\begin{equation}\label{eq:EXP-graded}
	\frac {\max\{\ha^{\vec1}:\phi\in\actives_{\omega}\}}{\min\{\ha^{\vec1}:\phi\in\actives_{\omega}\}}\le\Cgraded.
\end{equation}
\end{proposition}

\begin{proof}
First we rewrite $\cardinality{\Phi}$ as a sum of integrals
$$
\begin{aligned}
\cardinality{\Phi}&=\sum_{\phi\in\Phi}1 =\sum_{\phi\in\Phi}\int_{\esupp_\phi}\ha^{-\vec1} d\vx =\sum_{\phi\in\Phi}\sum_{\omega\subseteq\esupp_\phi}\int_{\omega}\ha^{-\vec1} d\vx
=\sum_{\omega\in\mesh}\int_{\omega}\sum_{\phi\in\actived_\omega}\ha^{-\vec1} d\vx.
\end{aligned}
$$
Since $\cardinality{\actived_\omega}\ge (\vd+\vec1)^{\vec1}$ and $\ha^{-\vec1}\ge\hm^{-\vec1}$ we have the lower bound.
Using \ref{H:duals} and \eqref{eq:EXP-graded} we have $\cardinality{\actived_\omega}\le\Cduals$ and $\ha^{-\vec1}\le\Cgraded\hm^{-\vec1}$ from which the upper bound follows.
\qed\end{proof}

We now show exponential convergence as a function of the space dimension provided that space resolution is small, but not too small.

\begin{corollary}\label{thm:EXP-convergence-dim}
Under the assumptions of Theorem~\ref{thm:EXP-convergence} and of Proposition~\ref{thm:EXP-cardinality}, if $\hat\tau<1$ and there are $a, b>0$ such that for all $d\ge \abs{\vsigma}+n$
\begin{equation}\label{eq:EXP-mesh-size-equiv}
\Cduals\Cgraded\le a (d+1)^n, \qquad \frac1b  \frac{\hat\tau \vmin{\vr} }{ 2 n^{\frac 32}\Bgrowth}\le  \vmin\hmd,\forall \vx\in\Omega,
\end{equation}
then there is $\tau_{\#}<1$ and $C_{\tau_{\#}}$ independent of $d$ and $\Phi_d$ such that
$$
\norm{\partial^{\vsigma}(f-\aop_d f)}_\infty\le
 C_{\tau_\#} \tau_{\#}^{(\cardinality{\Phi_d})^{1/n}}\norm{f}_{\infty, \Delta}.
$$
\end{corollary}

\begin{proof}
From \eqref{eq:EXP-mesh-size-equiv} we have
$$
\hmd^{-\vec1}\le (\vmin\hmd)^{-n} \le b^n \hat\tau^{-n}(2n^\frac32\Bgrowth)^n(\vmin\vr)^{-n}.
$$
Proposition~\ref{thm:EXP-cardinality} gives
$$
\cardinality{\Phi_d} \le a (d+1)^n \int_{\Omega}\hmd^{-\vec1} d\vx \le  t^{-n} (d+1)^n
$$
where $t^{-n}:= a b^n \hat\tau^{-n}(2 n^{\frac 32 } \Bgrowth)^n\int_{\Omega} ( \vmin{\vr(\vx)} )^{-n} \,d\vx.$
From the above we deduce $d+1\ge t(\cardinality{\Phi_d})^{\frac1n}$.
By \eqref{eq:EXP-clean}, and setting $\tau_{\#}:=\tau^t$ and $C_{\tau_\#}:=C_{\tau}\tau^{-\abs{\vsigma}}$, we have the result for all $1>\tau_\#>\hat\tau^t$. 
\qed\end{proof}

\section{B-splines and coefficient functionals}
\label{sec:b-splines}

Many of the spline spaces used in applications are generated by B-splines and in Section~\ref{sec:SP} we consider tensor product splines (TPS),  analysis suitable T-splines (AST), hierarchical splines (HS) and the locally refined splines (LR).
This section recalls some properties of B-splines and TPS that help in the construction of $\aop$ satisfying the assumptions in Sections~\ref{sec:framework},\dots,\ref{sec:EXP}.
For more material and proofs of some of the results below we refer to \cite{deboor_book}\cite{tom_notes}\cite{schumaker_book}.

We recall first the properties of B-splines that we need for \ref{H:polynomial} and \ref{H:phi}.
Then we introduce two families of coefficient functionals for TPS that satisfy \ref{H:lambda}.
The functional $S_\phi$ from \cite{scherer1999new} provides the smallest bound for $\Clambda$ available in the literature, but is restricted to $\support S_\phi=\support\phi$.
Another functional $G_{\phi, \eta}$ allows for the choice of $\eta=\support G_{\phi, \eta} \subseteq \support\phi$ and it is based on \cite{CIME}. 

\subsection{B-splines and their smoothness: \texorpdfstring{\ref{H:phi}}{H:phi}}

A univariate B-spline of degree $d$ is a compactly supported non-negative piecewise polynomial with minimal support.
Each B-spline of degree $d$ is uniquely associated to a \emph{knot vector}, a non-decreasing sequence of $d+2$ real numbers that encodes its smoothness properties and its polynomial subdomains.
The univariate B-spline $\varphi$ of degree $d$ associated with the knot vector $\Theta(\varphi):=[\theta_1(\varphi),\dots,\theta_{d+2}(\varphi)]$ has support $[\theta_{1}(\varphi),\theta_{d+2}(\varphi)]$ so that $h_\varphi=\theta_{d+2}(\varphi)-\theta_{1}(\varphi)$ and it is defined by
\def\bbl{\bm[\hspace{-.1cm}\bm[}
\def\bbr{\bm]\hspace{-.1cm}\bm]}
\begin{equation}\label{eq:bspline-def}
\varphi ( x ) := h_\varphi \bbl \theta_1(\varphi),\dots,\theta_{d+2}(\varphi) \bbr ( \cdot - x )_+^d,
\end{equation}
where $\bbl\theta_1,\dots,\theta_{d+2}\bbr f$ is the usual divided difference of $f$, that in this case is given by $f(\theta):=(\theta-x)_+^{d}$.

We recall that $\varphi|_Z\in\poly_d$ for all $Z$ of the form $[\theta_i(\varphi), \theta_{i+1}(\varphi)]$ with $ \theta_i(\varphi) < \theta_{i+1}(\varphi)$.
If $m_{\Theta_\varphi}(x)$ is the number of repetitions of $x$ in $\Theta(\varphi)$ then $\varphi$ has $d-m_{\Theta_\varphi}(x)$ continuous derivatives at $x$.

Note that from \eqref{eq:bspline-def} and properties of divided differences it follows that
\begin{equation}\label{eq:bspline-integral}
\int \varphi(x)\,dx= h_\varphi \bbl \theta_1(\varphi),\dots,\theta_{d+2}(\varphi) \bbr \int_{\support\varphi} ( \cdot - x )_+^d \,dx = \frac{ h_\varphi }{d+1}.
\end{equation}

A $n$-variate tensor product B-spline $\phi$ of degree $\vd=(d_1, \dots, d_n)$ is a product of $n$ univariate B-splines
\begin{equation}\label{eq:phi-tensor-product}
\phi(\vx):=\varphi_1(x_1)\cdots\varphi_n(x_n)
\end{equation}
and is defined by an $n$-tuple $\vTheta(\phi):=(\Theta(\varphi_1),\dots,\Theta(\varphi_n))$ of knot vectors.

\begin{lemma}\label{thm:b-spline-norm}
For all B-spline $\phi=\varphi_1\cdots\varphi_n$ of degree $\vd$ we have
\begin{equation}\label{eq:bspline-pnorm-estimate}
\frac{\hp^{\vec1/p}}{(\vd+\vec1)^{\vec1}}\le\norm{\phi}_p\le\frac{\hp^{\vec1/p}}{(\vd+\vec1)^{\vec1/p}},\qquad 1\le p\le\infty.
\end{equation}
\end{lemma}

\begin{proof}
\def\hpi#1{h_{\varphi_{#1}}}
By \eqref{eq:phi-tensor-product} we have $\norm{\phi}_p=\norm{\varphi_1}_p \cdots\norm{\varphi_n}_p$. By \eqref{eq:bspline-integral} and H\"older's inequality 
$$
1=\frac{d_i+1}{\hpi i} \int_{\R}\varphi_i(x_i)\,dx_i \le \frac{d_i+1}{\hpi i} \norm{\varphi_i}_p {\hpi i}^{1-1/p}
$$
and the lower bound follows.
Since $\norm{\varphi_i}_\infty\le 1$ we have
$$
\Big(\int_\R \varphi_i(x_i)^p \,dx_i\Big)^{1/p}\le \Big(\int_\R \varphi_i(x_i) \,dx_i\Big)^{1/p} = \frac{\hpi i^{1/p}}{(d_i+1)^{1/p}}
$$
giving the upper bound.
\qed\end{proof}

The $L^\infty$ norms of the derivatives of a B-spline depend on its knot vectors. For $\phi=\varphi_1\cdots\varphi_n$ let
\begin{equation}\label{eq:knot-differences2}
\Delta_{\phi, i, k}:=\min_{\ell= k+1, \dots, d_i+2}\{\theta_\ell(\varphi_i)-\theta_{\ell-k}(\varphi_i)\}.
\end{equation}
Then, as stated in the following proposition, \ref{H:phi} is implied by
\begin{enumerate}
\hsubitem{knots}{\Xi}{phi} $\displaystyle\frac{\Delta_{\phi, i, d_i+1}}{\Delta_{\phi, i, d_i-\sigma_i+1}}
\le\Cknots$, for all $\phi\in\Phi$, $i=1, \dots, n$, \emph{knot vector regularity}.
\end{enumerate}

\begin{proposition}\label{thm:OP-spline-reg}
Let $\Phi$ be a collection of B-splines of the same degree satisfying \ref{H:knots}.
Then  \ref{H:phi} holds with
\begin{equation}\label{eq:bspline-cphi}
\Cphi =\frac {(\vd+\vec1)!\, 2^{\abs{\vsigma}}}{(\vd-{\vsigma})!}\Cknots^{\abs{\vsigma}}.
\end{equation}
\end{proposition}

\begin{proof}
Using $\partial^{\vsigma}\phi(\vx) =\partial^{\sigma_1}_1\varphi_1(x_1)\, \cdots\, \partial^{\sigma_n}_n\varphi_n(x_n)$, 
 \ref{H:knots}, and the following univariate bound \cite[Theorem 4.22]{schumaker_book} 
\begin{equation}\label{eq:bspline-deriv-estimate}
\norm{\partial^{\sigma_i}\varphi_i}_\infty\le\frac {d_i!\, 2^{\sigma_i}}{(d_i-{\sigma_i})!}\prod_{k=d_i+1-{\sigma_i}}^{d}\Delta_{\phi,i, k}^{-1}\le\frac {d_i!\, 2^{\sigma_i}}{(d_i-{\sigma_i})!}\Delta_{\phi, i,d_i+1-{\sigma_i}}^{-{\sigma_i}},
\end{equation}
we have
$$
\norm{\partial^{\vsigma}\phi}_\infty\le\frac{\vd!2^{\abs{\vsigma}}}{(\vd-{\vsigma})!}\prod_{i=1}^n\Delta_{\phi, i, d_i+1-\sigma_i}^{-\sigma_i}%
\le \frac{\vd!2^{\abs{\vsigma}}}{(\vd-{\vsigma})!}\Cknots^{\abs{\vsigma}}\hp^{-{\vsigma}}.
$$
\ref{H:phi} follows using H\"older's inequality and the lower bound for $\norm{\phi}_p$ in \eqref{eq:bspline-pnorm-estimate}.
\qed\end{proof}

\subsection{Tensor product splines and polynomial representation}
\label{sec:OP-poly-rep-TPS}

A $(d+1)$-open knot vector is a non-decreasing sequence of real numbers $\Xi:=[\xi_1, \dots, \xi_{s+d+1}]$, $s\ge d+1$, with the following properties
\begin{align*}
\xi_1=\dots=\xi_{d+1}, &&\xi_{s+1}=\dots=\xi_{s+d+1}, &&\xi_{i+d+1}>\xi_i, \ i=1, \ldots, s.
\end{align*}
Associated with $\Xi$ is the space $\spline_{\Xi, d}\supseteq\poly_d$ of piecewise polynomials of degree $d$ on the mesh $\{[\xi_i, \xi_{i+1}]:\xi_i <\xi_{i+1}\}$ that have $d-m_\Xi(x)$ continuous derivatives at $x\in[\xi_1,\xi_{s+d+1}]$, where $m_\Xi( x)$ is the number of repetitions of $x$ in $\Xi$.
The B-splines $\{\varphi_1, \dots, \varphi_s\}$, where $\Theta(\varphi_i)=[\xi_i,\dots,\xi_{i+d+1}]$, constitute a basis of $\spline_{\Xi, d}$ and provide a non-negative partition of unity.

Similarly, with an $n$-tuple of knot vectors $\vXi:=(\Xi_1,\dots,\Xi_n)$ where each $\Xi_i:=[\xi_{i,1}, \dots, \xi_{i,s_i+d_i+1}]$ is a $d_i$-open knot vector we associate the tensor product spline space
$$\spline_{\vXi,\vd}:=\spline_{\Xi_1,d_1}\otimes\dots\otimes\spline_{\Xi_n,d_n}.
$$
The canonical basis has $\vs^{\vec 1}$ elements $\phi_{\vi}$, with $\vi\in[\vec1,\vs]\cap\N^n$
$$
\phi_{\vi}(\vx):=\varphi_{1,i_1}(x_1)\cdots\varphi_{n,i_n}(x_n),\qquad \Theta(\varphi_{j,i_j}):=[\xi_{j,i_j}, \dots, \xi_{j,i_j+d_j+1}].
$$
Any $g\in\poly_{\vd}$ can be written explicitly as linear combination of the $\phi_{\vi}$
\begin{equation}
\label{eq:OP-blossom}
g=\sum_{\vi\in[\vec1,\vs]}\blossom(g,\phi_{\vi})\phi_{\vi}.
\end{equation}
The coefficients $\blossom(g,\phi_{\vi})$ are sometimes called \emph{blossoms} and depend only on $g$ and the internal knots $\xi_{j,i_j+1},\dots,\xi_{j,i_j+d_j}$, $j=1,\dots,n$ of $\phi_{\vi}$.

\subsection{Two families of coefficient functionals}

A collection $\{S_\varphi\}$ of bi-orthogonal functionals for the B-spline basis of $\spline_{\Xi,d}$ was constructed in \cite{scherer1999new}.
The collection satisfies \ref{H:lambda} with $\Clambda=(d+1)2^{d+1}$.
No collection of bi-orthogonal functionals can satisfy \ref{H:lambda} with $\Clambda<k_{d,p}$ where $k_{d,p}$ is the condition number of the B-spline basis.
It is known, see \cite[p.528]{schumaker_book} for references, that $k_{d,p}\ge c (d+1)^{-1/p}2^{d+1}$ for some $c>0$ independent of $d$ and $p$. In this sense the functionals $S_\varphi$'s are close to optimal.

The functional in \cite{scherer1999new} is defined only for univariate B-splines, but it can easily be extended to the tensor product B-splines.
To a B-spline $\phi=\varphi_1\cdots\varphi_n$ we associate the operator $S_\phi$ defined by
\begin{equation}\label{eq:S-phi-def}
S_{\phi} (f):=\int_{\support\phi} w_{\varphi_1}(x_1)\cdots w_{\varphi_n}(x_n)f(\vx)\, d\vx,
\end{equation}
where the $w_{\varphi_i}$ are described in \cite{scherer1999new} and have the same support as $\varphi_i$.

The bi-orthogonality properties of the $S_\phi$, their support and their norm directly lead to the following proposition.

\begin{proposition}
\label{thm:OP-s-phi}
Let $\Phi$ be a collection of B-splines of degree $\vd$, then $\Lambda=\{S_{\phi}:\phi\in\Phi\}$ satisfies the assumptions \ref{H:lambda} and \ref{H:esupp}  with $\Clambda= (\vd+\vec1)^{\vec1} 2^{\abs{\vd}+n}$ and $\Cesupp=1$.
Moreover, if $\Phi$ is a TPS basis satisfying \ref{H:knots} then $\aop$ defined by \eqref{eq:operator-in-gen-form} is a projector onto $\SPAN\Phi$, and \ref{H:polynomial}, \ref{H:lambda}, \ref{H:phi} and \ref{H:esupp} are satisfied.
\end{proposition}

We will use the functional $S_\phi$ for TPS, AST, THB and a subclass of LR.
In the next subsection and for general LR we need a functional with a smaller support.
We obtain it by modifying the construction in \cite{CIME}.

\begin{definition}
Let $\phi=\varphi_1\cdots\varphi_n$ be a B-spline of degree $\vd$ and $\eta=\eta_1\times\dots\times\eta_n\subseteq\support\phi$ be a box.
We define
\begin{equation}
\label{eq:def-g-phi-eta}
G_{\phi, \eta} (f):=\frac{1}{\measure{\eta}}\int_{\eta} w_{\varphi_1, \eta_1}(x_1)\cdots w_{\varphi_n, \eta_n}(x_n) f(\vx)\, d\vx
\end{equation}
where $w_{\varphi_i, \eta_i}\in\poly_{d_i}$ are such that
\begin{equation}
\label{eq:OP-poly-coefs}
G_{\phi, \eta}(g)=\blossom(g,\phi), \, \qquad\forall g\in\poly_{\vd}.
\end{equation}
\end{definition}

\begin{proposition}
\label{thm:OP-g-phi-eta}
For all $\phi$ and $\eta\subseteq\support\phi$, $G_{\phi,\eta}$ is well defined.
If $\Phi$ is a collection of B-splines of degree $\vd$, then for $\Lambda=\{G_{\phi,\eta_\phi}: \eta_\phi\subseteq\support\phi,\, \phi\in\Phi\}$ the assumptions \ref{H:lambda} and \ref{H:esupp} hold with constants $\Cesupp = 1$ and
$$
\Clambda = (\vd+\vec1)^{\vec1-\vec1/p} \max_{\phi\in\Phi} \{{\hp^{\vd+\vec1/p} \vh_{\eta_\phi}^{-\vd-\vec1/p}}\} \ \prod_{i=1}^n \onorm{H_{d_i}^{-1}}.
$$
Here $H_d$ is the Hilbert matrix of order $d+1$, i.e., the element in position $(s,t)$ is  $(s+t-1)^{-1}$, and $\onorm{\cdot}$ is the $\ell^\infty$ operator norm.
Moreover, if $\Phi$ is a TPS basis satisfying \ref{H:knots} then with $\aop$ defined by \eqref{eq:operator-in-gen-form}, \ref{H:polynomial}, \ref{H:lambda}, \ref{H:phi} and \ref{H:esupp} are satisfied.
If additionally each $\eta_\phi$ is contained in some $\omega_\phi\in\mesh$, then $\aop$ is a projector onto $\SPAN\Phi$.
\end{proposition}

\begin{proof}
Due to the tensor product structure, all claims follows from the univariate case.
Let $\varphi$ be a B-spline of degree $d$ and $\eta\subseteq\support\varphi$ be the desired support. 

First we show that $w_{\varphi,\eta}$ is determined by \eqref{eq:OP-poly-coefs}.
Let $t:=\frac{x-\min\eta}{\measure{\eta}}$, then expressing $w_{\varphi, \eta}(x)$ as 
$$w_{\varphi, \eta}(x):=\sum_{j=0}^d c_j t(x)^j$$
we find
$$
G_{\varphi, \eta}(t^i)=\sum_{j=0}^d c_j\frac{1}{\measure{\eta}}\int_{\eta} t(x)^{i+j}\, dx=\sum_{j=0}^d \frac{c_j}{i+j+1}.
$$
Thus \eqref{eq:OP-poly-coefs} is equivalent to $\vc:=(c_1, \ldots, c_{d+1})$ satisfying
\begin{align}
H_d\vc =\vb,
\end{align}
where $\vb:=(\blossom(t^0,\varphi), \ldots,\blossom(t^d,\varphi))$.
Since $H_d$ is square and invertible the system is well posed and $w_{\varphi,\eta}$ is uniquely determined. 

We now prove \ref{H:lambda}. 
By H\"older's inequality, and $0\le t( x)\le 1$ for $ x\in\eta$,  we have
\begin{equation}\label{eq:op-G-phi-eta-1}
\norm{G_{\varphi, \eta}}_{*p}\le\frac{\norm{w_{\varphi, \eta}}_{p', \eta}}{\measure{\eta}}
\le \frac{\norm{w_{\varphi, \eta}}_{\infty, \eta}}{\measure{\eta}^{1/p}}\le \frac{(d+1)}{\measure{\eta}^{1/p}}\onorm{H_d^{-1}}\norm{\vb}_{\ell^\infty}.
\end{equation}
To bound $\norm{\vb}_{\ell^\infty}$ we use Marsden's identity
\begin{align}\label{eq:OP-marsden}
\blossom((y-x)^d,\phi ) = \prod_{i=2}^{d+1}(y-\theta_i(\phi))
\end{align}
to obtain 
$$\abs{\blossom(t^r,\phi )}\le\frac1{(d-r)!} \frac{\measure{\support\varphi}^r}{\measure{\eta}^{r}}\le \frac{\measure{\support\varphi}^d}{\measure{\eta}^{d}},\qquad r=0,\dots,d.$$
Now \ref{H:lambda} for univariate B-splines follows from \eqref{eq:op-G-phi-eta-1} and \eqref{eq:bspline-pnorm-estimate}.

Finally if $\Phi$ is a TPS basis and $\eta_\phi\subseteq\omega_\phi\in\mesh$ then by \eqref{eq:OP-blossom} we have 
$G_{\phi,\eta_\phi}(\psi)= \blossom(\psi|_{\omega_\phi},\phi)=\delta_{\phi,\psi}$ for all  $\psi\in\Phi$, and consequently $\aop$ is a projector.
\qed\end{proof}

\section{Space specific results}
\label{sec:SP}

In this section we describe a specific approximation operator $\aop$ for TPS, AST, THB and LR splines.
For each operator we show which of the abstract assumptions of Section~\ref{sec:framework} hold and give an upper bound for the corresponding constants.
Combining this information with Theorem~\ref{thm:sobolev} one obtains error bounds with Sobolev or reduced seminorms. 

\subsection{Application to tensor product B-splines}
\label{sec:tp-splines}

Assuming the usual mesh regularity, i.e. \ref{H:mesh-length} and \ref{H:isotropicA}, we obtain all the error bounds in Theorem~\ref{thm:sobolev} for TPS.
Compared to what can be found in the literature, see for example \cite{local-spline-approx}\cite{ANU:9260759}\cite{MR2877950}, we take into account the local mesh size for all combinations of $p$ and $q$.

\begin{theorem}\label{thm:tp-splines}
Let $\Phi$ be a TPS basis satisfying \ref{H:knots}. Then the operator $\aop$ corresponding to $\Lambda=\{\lambda_\phi=S_\phi:\, \phi\in\Phi\}$ is a projector onto $\spline=\SPAN\Phi$ and the assumptions \ref{H:polynomial}, \ref{H:lambda}, \ref{H:phi}, \ref{H:esupp}, \ref{H:duals}, \ref{H:mesh-num} and \ref{H:concave} are satisfied with 
$$
\begin{aligned}
&\Clambda=(\vd+\vec1)^{\vec1} 2^{\abs{\vd}+n},&&\Cphi =\frac {(\vd+\vec1)!\, 2^{\abs{\vsigma}}}{(\vd-{\vsigma})!}\Cknots^{\abs{\vsigma}},&&\Cesupp=1,
\\&\Cduals=(\vd+\vec1)^{\vec1},&&\Cnum=(\vd+\vec1)^{\vec1},&&\Cconcave=1.
\end{aligned}$$
\end{theorem}

\begin{proof}
	For TPS \ref{H:mesh-num} holds with $\Cnum=(\vd+\vec1)^{\vec1}$, and since $\support\phi=\esupp_\phi$ we have \ref{H:duals} with $\Cduals=(\vd+\vec1)^{\vec1}$.
	Then the thesis follows from Proposition~\ref{thm:OP-spline-reg}, Proposition~\ref{thm:OP-s-phi}.
\end{proof}

%

\subsection{Application to AST-splines}
\label{sec:ast-splines}

Cubic T-splines were introduced in \cite{sederberg2004t}\cite{sederberg2003t} for geometric modelling applications.
The idea was to reduce the number of  control points by replacing the control polygon of TPS with a T-mesh.
Depending on $\vd$, a tensor product B-spline is associated with each vertex, edge or element of the T-mesh.
The T-spline space is spanned by these B-splines.

\emph{Analysis Suitable T-splines} (AST) avoid linear dependencies by restricting the class of allowed T-meshes \cite{MR3084741}\cite{MR2871362}\cite{MR3003061}.
AST spaces can be constructed in 2D \cite{MR3341848} and are also defined for 3D domains \cite{MR3519556}.
In particular AST are dual compatible, cf. 
 \cite{MR3084741}\cite{MR3003061} and bi-orthogonal functionals for TPS are bi-orthogonal to AST.
 
By Proposition~\ref{thm:OP-spline-reg} and and the bi-orthogonality properties of $S_\phi$
we obtain the following theorem.
Compared to the result for TPS there is no a priori bound for $\Cnum$ which can be used as a measure of the mesh complexity.

\begin{theorem}\label{thm:ast-splines}
Let $\Phi$ be an AST basis satisfying \ref{H:knots}.
The operator $\aop$ corresponding to $\Lambda=\{\lambda_\phi=S_\phi:\, \phi\in\Phi\}$ is a projector onto $\spline=\SPAN\Phi$ and  \ref{H:polynomial}, \ref{H:lambda},\ref{H:phi}, \ref{H:esupp}, \ref{H:duals} and \ref{H:concave} hold with 
$$
\begin{aligned}
&\Clambda=(\vd+\vec1)^{\vec1} 2^{\abs{\vd}+n},&&\Cphi=\frac{(\vd+1)! 2^{\abs{\vsigma}}}{(\vd-{\vsigma})!}\Cknots^{\abs{\vsigma}},
\\&\Cesupp=1,&&\Cduals=(\vd+\vec1)^{\vec1},&&\Cconcave=1.
\end{aligned}$$
\end{theorem}

\subsection{Application to THB splines}
\label{sec:hb-splines}

Hierarchical spline were introduced in \cite{Forsey:1988:HBR:378456.378512} \cite{kraft1997}.
Quasi-interpolants have been constructed in \cite{MR3627466}\cite{MR3439218} using the \emph{truncated basis} \cite{MR2925951}. See \cite{CIME} for a recent survey.

\def\hbasis{\Psi}
\def\hbas{\psi}
\def\bas{\phi}
\def\numSub{s}
\def\hsubdomain{\Omega}
\def\trunc{\mathtt{T}}
\def\hbb{\mathbb{H}}

Let $\hbasis_1, \dots, \hbasis_\numSub$ be a sequence of TPS bases  that span nested spaces, i.e. $i<j\thus\SPAN\hbasis_i\subset\SPAN\hbasis_j$ and $\Omega=\hsubdomain_1\supseteq\dots\supseteq\hsubdomain_\numSub=\emptyset$ a corresponding sequence of closed domains.
The hierarchical basis $\hbb$ is as follows
\begin{equation}
\label{eq:HBselection}
\hbb :=\bigcup_{i=1}^\numSub\{\hbas \in\hbasis_i :
\support\hbas \subseteq\hsubdomain_i\mbox{ and }
\support\hbas \cap (\hsubdomain_i
\setminus\hsubdomain_{i+1})\neq\emptyset\}.
\end{equation}
The associated box-mesh $\mesh$ contains a similar selection of elements from the tensor product meshes $\mesh_1,\dots,\mesh_s$ corresponding to $\hbasis_1,\dots,\hbasis_s$:
$$
\mesh:=\bigcup_{i=1}^\numSub\{\omega \in\mesh_i :
\omega\in \Omega_i\setminus\Omega_{i+1}\}.
$$
The truncation operator $\trunc_i:\SPAN\hbasis_i\to\SPAN\{\psi\in\hbasis_i:\hbas|_{\Omega\setminus\hsubdomain_i}\neq 0\}$ is defined as
\begin{equation}\label{eq:HB-truncation}
\trunc_i \Big(\sum_{\hbas \in\hbasis_i} c_{\hbas}\hbas\Big) :=\sum_{\substack{\hbas \in\hbasis_i :\\\hbas|_{\Omega\setminus\hsubdomain_i}\neq 0}} c_{ \hbas}\hbas.
\end{equation}
By recursive truncation one obtains the \emph{truncated basis}
\begin{equation*}
\hbb^\trunc :=\{\trunc_{\numSub}\cdots\trunc_{i+1}\hbas :
\hbas \in\hbb\cap\hbasis_i,\, i=1,\dots,\numSub\}.
\end{equation*}
It is convenient to abbreviate $\trunc_{\numSub}\cdots\trunc_{i+1}\hbas$ with $\trunc\hbas$ and to annotate the symbols referring to $\hbb^\trunc$ with a superscript $\trunc$ to distinguish them from those referring to $\hbb$.
One of the advantages of the truncated basis is that the coefficients of a polynomial $g\in\poly_{\vd}$ are the same as for TPS
\begin{equation}
\label{eq:THB-coefficients}
g=\sum_{\phi\in\hbb}\blossom(g,\phi)\trunc\phi.
\end{equation}

We say that a HB basis is $k$-\emph{admissible} if for each $\omega\in \mesh_j\cap\mesh$, $\actives_\omega$ contains only functions $\phi\in\hbasis_i$ for $j-k<i\le j$.
A THB basis is $k$-admissible if the corresponding HB basis is.
See also \cite{MR3417723}\cite{MR3545984}.

\begin{theorem}\label{thm:thb-splines}
Let $\Phi$ be a $k$-\emph{admissible} THB basis satisfying \ref{H:phi}, and $a$ be such that $\hp\le a\ho$, for all $\phi\in\hbasis_i$ and $\omega\subseteq\support \phi$, such that $\omega\in\mesh_i\cap\mesh$.
Then the operator $\aop$ corresponding to
$\Lambda=\{\lambda_{\trunc\phi}= S_\phi:\trunc\phi\in\Phi\}$ is a projector onto $\spline=\SPAN\Phi$, and  \ref{H:polynomial}, \ref{H:lambda}, \ref{H:esupp}, \ref{H:duals} and \ref{H:concave} hold with 
$$\Clambda=(\vd+\vec1)^{\vec1} 2^{\abs{\vd}+n},\qquad\Cesupp=a,\qquad\Cduals=k(\vd+\vec1)^{\vec1},\qquad\Cconcave=1.$$
\end{theorem}
 
\begin{proof}
\ref{H:polynomial} follows from \eqref{eq:THB-coefficients}. Since $\norm{\trunc\phi}_p\le\norm{\phi}_p$ we have \ref{H:lambda} with $\Clambda=(\vd+\vec1)^{\vec1} 2^{\abs{\vd}+n}$. 
Noting that $\esupp_{\trunc\phi}=\support S_\phi=\support\phi$ we have
$\vh_{\esupp_{\trunc\phi}}=\hp\le a\ho\le a\vh_{\trunc\phi}$ and \ref{H:esupp} follows with $\Cesupp=a$.
As $\Phi$ is $k$-admissible and $\esupp_{\trunc\phi}=\support\phi$ we have $\actived_\omega^{\trunc}=\actives_\omega$ and consequently  $\cardinality{\actived_\omega^\trunc}\le k (\vd+\vec1)^{\vec1}$.
Since $\esupp_{\trunc\phi}$ is a box we have $\Cconcave=1$.
\qed\end{proof}

Note that THB require more assumptions than TPS and AST.
Comparing with \cite{MR3627466}\cite{MR3439218}, we obtain global estimates with different $p$, $q$ norms that take into consideration the local mesh size, see Theorem~\ref{thm:sobolev}.

\subsection{Application to LR-splines}
\label{sec:lr-splines}

LR-splines where introduced in \cite{MR3019748}, here we use the equivalent definition from \cite{MR3146870}.
A box mesh $\mesh$ and a function $m$ that assigns a nonnegative integer to each interface between two elements define the spline space
$$\ppoly:=\left\{\begin{aligned}
f\in L^\infty(\Omega) :  &f|_{\omega}\in\poly_{\vd}, &&&\forall\omega\in\mesh,\\
&f\in\smooth^{d_i-m(E)}(\interior E), &&& \text{ for all inter-element interfaces }E
\end{aligned} \right\},$$
where $i$ is the direction of the normal to $E$. Here we assume that $\Omega$ is a box.

Given a B-spline $\phi=\varphi_1\cdots\varphi_n$ and $s\in\{1,\dots,n\}$ we can insert a knot $\bar\theta\in(\theta_1(\varphi_s),\theta_{d+2}(\varphi_s))$ in $\Theta(\varphi_s)$
 and obtain the two B-splines $\hat\phi=\varphi_1\cdots\hat\varphi_s\cdots\varphi_n$ and $\tilde\phi=\varphi_1\cdots\tilde\varphi_s\cdots\varphi_n$ whose knot vectors are  $\Theta(\hat\varphi_s) =[\theta_1(\varphi_s),\dots,\bar\theta,\dots,\theta_{d+1}(\varphi_s)]$ and $\Theta(\tilde\varphi_s)  =[\theta_2(\varphi_s),\dots,\bar\theta,\dots,\theta_{d+2}(\varphi_s)]$, respectively.
We have
\begin{equation}\label{eq:LR-knot-insertion}
\phi=a \hat\phi + b \tilde\phi, \qquad a,b\in(0,1].
\end{equation}

Knot insertion defines a partial ordering on the set of B-splines, we write
$\phi\prec\psi$, if there exists a
sequence of B-splines $\phi_1, \dots, \phi_r$ with $\psi =\phi_1$ and $\phi=\phi_r$ such that
each $\phi_{i+1}$ is obtained from $\phi_i$ by knot insertion, see Fig.~\ref{figure:example-of-nested}.
If $\phi\prec\psi$ we define $$c_{\psi,\phi}:=\prod_{i=2}^r a_i,\qquad \text{where } \phi_{i-1}=a_{i}\phi_{i}+b_i \tilde\phi_{i}\text { as in \eqref{eq:LR-knot-insertion}.}
$$
We write $\phi\prec_\ppoly\psi$ if there exists a similar sequence whose elements are all contained in $\ppoly$.
The minimal B-splines for $\prec_\ppoly$ are called \emph{minimal support} B-splines in $\ppoly$.

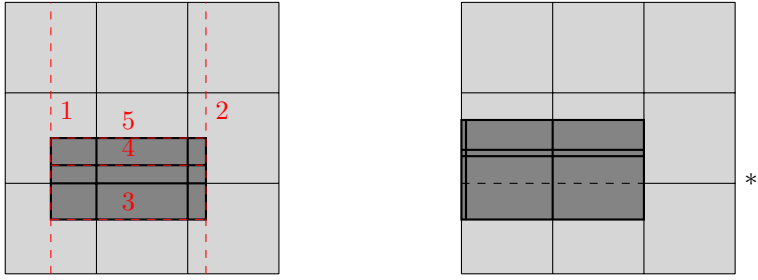
\begin{figure}
\begin{center}
\def\splined#1#2#3#4#5#6#7#8#9{%
\draw [#9, thick]%
	(#1, #5) rectangle (#4, #8)%
	(#2, #5)--(#2, #8)%
	(#3, #5)--(#3, #8)%
	(#1, #6)--(#4, #6)%
	(#1, #7)--(#4, #7);%
}
\begin{tikzpicture}[scale=1.2]
\fill[\FF] (0, 0) rectangle (3, 3);
\draw[] (0, 0) grid (3, 3);
\splined{.5}{1}{2}{2.2}{.6}{1}{1.2}{1.5}{fill=\RF}
\draw[dashed,red] (2.2,0)--(2.2,3) node[pos=.6,right] {2};
\draw[dashed,red] (.5,0)--(.5,3) node[pos=.6,right] {1};
\draw[dashed,red] (.5,.6)--(2.2,.6) node[midway,above] {3};
\draw[dashed,red] (.5,1.2)--(2.2,1.2) node[midway,above] {4};
\draw[dashed,red] (.5,1.5)--(2.2,1.5) node[midway,above] {5};



\begin{scope}[xshift=5cm]
\fill[\FF] (0, 0) rectangle (3, 3);
\draw[] (0, 0) grid (3, 3);
\splined{0}{0.05}{1}{2}{.6}{1.3}{1.37}{1.7}{fill=\RF}
\node[right] at (3,1) {*};
\draw[dashed] (0,1)--(2,1);
\end{scope}
\end{tikzpicture}\end{center}
\caption[]{
Two examples with B-splines of degree $(2,2)$.
On the left, $\phi$(\tikz{\fill[\RF] (0, 0) rectangle (.3, .3);}) is nested in $\psi$(\tikz \fill[\TF](0, 0) rectangle (.3, .3);). The dashed red lines represent a possible sequence of knots insertions to obtain $\phi$ form $\psi$.
On the right, $\phi$ is not nested in $\psi$ even if $\support\phi\subseteq\support\psi$ because the marked line corresponds to a knot of $\psi$, but not to a knot of $\phi$.}\label{figure:example-of-nested}
\end{figure}

Since $\poly_{\vd}\subseteq\ppoly$, the space $\ppoly$ contains the Bernstein polynomials $B_{\vi}$, $\vi\in[\vec0,\vd]\cap\N^n$ on $\Omega$. The LR generating set $\Phi$ associated to $\ppoly$ contains the minimal support B-splines $\phi\in\ppoly$ that are obtained from the $B_{\vi}$'s using knot insertion.
The B-splines in $\Phi$ span a subset of $\ppoly$ containing $\poly_{\vd}$, but they can be linearly dependent.
\def\lrc{\mathfrak{l}}
If $g\in\poly_{\vd}$ then we have $g=\sum_{\phi\in\Phi}\lrc(g,\phi)\phi$, where, with $\blossom(g,\phi)$ as in \eqref{eq:OP-blossom}, the coefficients are given by the following recursive formula \cite[Theorem 4.4]{MR3146870} 
\begin{equation}
\label{eq:coefficients-lr}
\lrc(g,\phi):=\blossom(g,\phi)-\sum_{\psi\in N_\phi} c_{\psi, \phi} \lrc(g,\psi),\qquad N_\phi:=\{\psi\in\Phi:\phi\prec\psi\}.
\end{equation}
Unfolding the recursion we find
\begin{equation}\label{eq:coefficients-lr2}
\lrc(g,\phi)=\blossom(g,\phi)+\sum_{\psi\in N_\phi} z_{\psi,\phi}\blossom(g,\psi)
\end{equation}
 for some $z_{\psi,\phi}\in \R$.
It follows that \ref{H:polynomial} is satisfied by the collection $\Lambda=\{\lambda_\phi, \, \phi\in\Phi\}$, 
\begin{equation}
\label{eq:lr-lambda}
\lambda_\phi := S_\phi+\sum_{\psi\in N_\phi} z_{\psi,\phi} G_{\psi, \support\phi},
\end{equation}
where $S_\phi$  and $G_{\psi,\support\phi}$ are defined in \eqref{eq:S-phi-def} and \eqref{eq:def-g-phi-eta}, respectively.

\begin{theorem}\label{thm:lr-splines}
Let $\Phi$ be an LR generating system satisfying \ref{H:knots}, \ref{H:duals} and such that for all $\phi\in\Phi$, $\psi\in N_\phi$ we have $\vh_\psi\le \ell\hp$.
Then the collection $\Lambda=\{\lambda_\phi,\phi\in\Phi\}$ with $\lambda_\phi$ as in \eqref{eq:lr-lambda} defines an operator $\aop$ that satisfies \ref{H:polynomial}, \ref{H:lambda},\ref{H:phi} and \ref{H:esupp} with 
$$
\begin{aligned}
\Clambda=2^{\Cactives+\abs{\vd+\vec1}} (\vd+\vec1)^{\vec1-\vec1/p}\Cactives\ell^{\abs{\vd}+\frac np}\prod_{i=1}^n\onorm{H_{d_i}^{-1}},
\\\Cphi=\frac{(\vd+1)!\, 2^{\abs{\vsigma}}}{(\vd-{\vsigma})!}\Cknots^{\abs{\vsigma}},\qquad\Cesupp=1,\qquad \Cconcave=1.
\end{aligned}
$$
Moreover, if $\Cactives =(\vd+\vec1)^{\vec1}$, then $\Clambda=(\vd+\vec1)^{\vec1} 2^{\abs{\vd}+n}$.
\end{theorem}

\begin{proof}
\ref{H:polynomial} follows from \eqref{eq:lr-lambda}. Proposition~\ref{thm:OP-spline-reg} implies \ref{H:phi}.
Since $\esupp_\phi$ is a box we deduce \ref{H:esupp} with $\Cesupp=1$ and \ref{H:concave} with $\Cconcave=1$.
We need to show \ref{H:lambda}.
From \eqref{eq:lr-lambda} we have
\begin{equation}
\label{eq:proof-lr-continuity-3}
\norm{\lambda_\phi}_{*p} \le\big (\norm{S_\phi}_{*p}+\sum_{\psi\in N_\phi}\abs{z_{\phi, \psi}}\norm{G_{\psi, \support\phi}}_{*p}\big ).
\end{equation}
The $z_{\psi, \phi}$ are sums of products of $c_{\eta, \beta}$. Consequently they are sums of terms in $[-1,1]$. Therefore $\abs{z_{\psi, \phi}}$ is bounded by the number of terms and we have
\begin{equation}
\label{eq:proof-lr-continuity-2}
\abs{z_{\psi, \phi}}\le \cardinality{\{\text{oriented paths in $N_\phi$ from $\psi$ to $\phi$}\}}\le 2^{\cardinality{N_\phi}-1}\le 2^{\Cactives-2}.
\end{equation}
Finally, $\phi\prec\psi$ implies $\norm{\phi}_p\le \norm{\psi}_p$, and using \eqref{eq:proof-lr-continuity-3}, \eqref{eq:proof-lr-continuity-2}, Proposition~\ref{thm:OP-g-phi-eta} and $\vh_\psi\le\ell\hp$ we obtain \ref{H:lambda} with the claimed $\Clambda$.

If $\Cactives=(\vd+\vec1)^{\vec1}$ then by \cite[Theorem 4]{MR3370382}, $N_\phi=\emptyset$ and $\lambda_\phi=S_\phi$.
\qed\end{proof}

\section{Closing remarks}
\label{sec:remarks}

The polynomial approximation assumption \ref{H:approx} can be extended to allow for more general bounding terms.
A possibility are fractional order Sobolev seminorms, or derivative dependent summability $q_{\vk}$ as in \cite{dupont_scott}.
We do not foresee major difficulties in such extensions, but surely they would be notationally heavy.

A second remark is that the results are based on local bounds. In particular the cardinality of $\Phi$ and the boundedness of $\Omega$ are not used in the proofs. Note however that the embedding of $L^q(\Omega)\subseteq L^p(\Omega)$  for the case $p\le q $ requires $\measure{\Omega}<\infty$.

\section*{Acknowledgement}
The first author has received funding from the European Research Council under the European Union's Seventh Framework Programme (FP7/2007-2013) / ERC grant agreement 339643.

\begin{table}[ht]
\renewcommand{\arraystretch}{1.9}
\begin{tabular}{lll}
\multicolumn{3}{l}{Assumptions for space independent estimates}\\
 \ref{H:polynomial} &\emph{polynomial reproduction}& $\aop g=g$, $\forall g\in\poly$\\
 \ref{H:lambda} &\emph{functional continuity}& $\norm{\lambda_\phi}_{*p}\norm{\phi}_p\le\Clambda$\\
 \ref{H:phi} &\emph{generators' regularity} & $\norm{\partial^{\vsigma}\phi}_{p, \omega}\le\Cphi\hp^{-{\vsigma}}{\measure{\omega}^{\frac1p}}{\measure{\support\phi}^{-\frac1p}}\norm{\phi}_p$\\
 \ref{H:esupp} &\emph{locality of $\aop$}& $\ha\le\Cesupp\hp$ and $\ha^{\vec1}\le\Cesupp^n\mu(\support\phi)$\\
 \ref{H:approx} && {$
\norm{\partial^{\vbeta}(f-\paop_\omega f)}_{p, \eta}\le\Capprox\dfrac{\ho^{{\vgamma}}}{\he^{{\vgamma}}}\sum_{\vk\in K_{\vbeta}}\he^{\vk-{\vbeta}+\vpq}\norm{\partial^{\vk} f}_{q, \eta}$}\\
 \ref{H:duals} &\emph{bound on $\esupp_\phi$ overlaps}& $\cardinality{\actived_\omega}\le\Cduals$\\
 \ref{H:mesh} &\emph{mesh regularity}& $\Gamma_\phi:=\big(\sum_{\omega\in\contained_\phi}\ho^{{\vgamma}p+\vec1}\ha^{-{\vgamma} p-\vec1}\big)^{1/p}\le\Cmesh $\\
 \ref{H:isotropicO} &\emph{element shape regularity}& $\vmax\ho\le\CisotropicO\vmin\ho$\\
 \ref{H:isotropicA} &\emph{shape regularity of $\esupp_\phi$}& $\vmax\ha\le\CisotropicA\vmin\ha$\\
\hline
\multicolumn{3}{l}{Mesh regularity assumptions that imply \ref{H:mesh}}\\
 \ref{H:mesh-num} &\emph{$\#$elements in $\support\phi$}& $\cardinality{\contained_\phi}\le\Cnum$\\
 \ref{H:mesh-length}&\emph{mesh quasi uniformity}& $\ha\le\Clength\ho$\\
\hline
\multicolumn{3}{l}{Assumptions that imply \ref{H:approx} using averaged Taylor expansion}\\
 \ref{H:concave} & & $\esupp_\phi$ star shaped and  $\ha\le\Cconcave\hz$\\
\hline
\multicolumn{3}{l}{B-spline specific assumption}\\
 \ref{H:knots} & \emph{knot vector regularity} & $\Delta_{\phi, i, d_i+1}\Delta_{\phi, i, d_i-\sigma_i+1}^{-1}\le\Cknots $\\
\end{tabular}
\renewcommand{\arraystretch}{1}
\caption{Table of the assumptions and corresponding constants.}\label{table:hypothesis}
\end{table}

\bibliographystyle{abbrv}
\bibliography{lr_approx}
\end{document}